\documentclass[11pt]{amsart}
\usepackage[foot]{amsaddr}
\usepackage{ifxetex}
\ifxetex
\usepackage[no-math]{fontspec}
\else
\fi
\usepackage{amsmath}
\usepackage{amsfonts}
\usepackage{amssymb}

\usepackage{amsthm}
\usepackage{fullpage}
\usepackage[lining,semibold,type1]{libertine} % a bit lighter than Times--no osf in math
\usepackage[T1]{fontenc} % best for Western European languages
\usepackage{textcomp} % required to get special symbols
\usepackage[varqu,varl]{inconsolata}% a typewriter font must be defined
\usepackage[libertine,vvarbb]{newtxmath}
\usepackage[scr=rsfso]{mathalfa}
\usepackage{bm}
\usepackage{dsfont}

\usepackage{listings}
\usepackage{hyperref, color}
\hypersetup{colorlinks=true,citecolor=blue, linkcolor=blue, urlcolor=blue}
\usepackage[linesnumbered,boxed,ruled,vlined]{algorithm2e}
\usepackage{bm}
\usepackage{bbm}
\usepackage[numbers]{natbib}
\usepackage{xcolor}
\usepackage{enumerate} 
\usepackage{enumitem}
\usepackage{tabularx}
\usepackage{array}
\usepackage{cleveref}
\usepackage{mathrsfs}
\newcommand{\PGD}{P^{\mathrm{GD}}}

\newcolumntype{L}[1]{>{\raggedright\arraybackslash}p{#1}}
\newcolumntype{C}[1]{>{\centering\arraybackslash}m{#1}}
\newcolumntype{R}[1]{>{\raggedleft\arraybackslash}p{#1}}
  
\usepackage{makecell}

\usepackage{footnote}
\makesavenoteenv{tabular}

\newcommand{\DTV}[2]{d_{\mathrm{TV}}\left({#1},{#2}\right)}

\newcommand{\e}{\mathrm{e}}

\renewcommand{\epsilon}{\varepsilon}

\newcommand{\0}{-1}
\newcommand{\1}{+1}
\newcommand{\Bin}{\mathrm{Bin}}

\newcommand{\tmix}{T_{\textsf{mix}}}

\newcommand{\Cma}{\Psi^{\mathrm{cor}}}
\newcommand{\Ima}{\Psi^{\mathrm{inf}}}

\newtheorem{theorem}{Theorem}[section]

\newtheorem{claim}[theorem]{Claim}
\newtheorem*{claim*}{Claim}

\newtheorem{fact}[theorem]{Fact}
\newtheorem{lemma}[theorem]{Lemma}
\newtheorem{proposition}[theorem]{Proposition}

\theoremstyle{definition}

\newtheorem{definition}[theorem]{Definition}

\newtheorem*{remark*}{Remark}

\newcommand{\norm}[1]{\left\Vert#1\right\Vert}

\newcommand{\tuple}[1]{\left(#1\right)} 
\newcommand{\inner}[2]{\left\langle #1,#2\right\rangle}
\newcommand{\tp}{\tuple}

\newcommand{\abs}[1]{\left\vert#1\right\vert}
\newcommand{\ctp}[1]{\left\lceil#1\right\rceil}

\def\*#1{\boldsymbol{#1}} % Use \*A for \mathbf{A}
\def\+#1{\mathcal{#1}} % Use \+A for \mathcal{A}
\def\-#1{\mathrm{#1}} % Use \-A for \mathrm{A}
\def\^#1{\mathscr{#1}} % Use \^A for \mathscr{A}

\usepackage{todonotes}

\usepackage{xifthen}
\newcommand{\KL}[2]{D_{\mathrm{KL}}\tp{{#1}\,\Vert\, {#2}}}

\renewcommand{\todo}[1]{\typeout{TODO: \the\inputlineno: #1}\textbf{{\color{red}[[[ #1 ]]]}}}

\DeclareMathOperator*{\oPr}{\mathbf{Pr}}
\renewcommand{\Pr}[2][]{ \ifthenelse{\isempty{#1}}
  {\oPr\left[#2\right]}
  {\oPr_{#1}\left[#2\right]} } % Use \Pr[a]{b} for \mathbf{Pr}_a[b], \Pr{b} for  \mathbf{Pr}[b]

\DeclareMathOperator*{\oE}{\mathbf{E}}
\newcommand{\E}[2][]{ \ifthenelse{\isempty{#1}}
  {\oE\left[#2\right]}
  {\oE_{#1}\left[#2\right]} }

\def\oVar{\mathbf{Var}}
\newcommand{\Var}[2][]{ \ifthenelse{\isempty{#1}}
  {\oVar\left[#2\right]}
  {\oVar_{#1}\left[#2\right]} }

\def\oEnt{\mathbf{Ent}}
\newcommand{\Ent}[2][]{ \ifthenelse{\isempty{#1}}
  {\oEnt\left[#2\right]}
  {\oEnt_{#1}\left[#2\right]} }

\newcommand{\EE}[2][]{ \ifthenelse{\isempty{#1}}
  {\mathbf{E}\left[#2\right]}
  {\mathbf{E}_{#1}\left[#2\right]} }

\def\mEnt{\mathbf{Cov}}
\newcommand{\MEnt}[2][]{ \ifthenelse{\isempty{#1}}
  {\mEnt\left(#2, \log #2\right)}
  {\mEnt_{#1}\left(#2, \log #2\right)} }

\newcommand{\Rd}{\mathsf{Trans}}
\newcommand{\relaxT}[2][]{
  \ifthenelse{\isempty{#2}}
  {t_{\mathrm{rel}}^{\mathrm{#1}}}
  {t_{\mathrm{rel}}^{\mathrm{#1}}(#2)}
}

\usepackage{xargs}
\newcommandx{\HyperGeo}[3][1 = \ell, 2=V, 3=k]{\Pi_{#2, #3, #1}}

\title{Optimal Mixing Time for the Ising Model in the Uniqueness Regime}
\author{Xiaoyu Chen}
\author{Weiming Feng}
\author{Yitong Yin}
\author{Xinyuan Zhang}

\address[Xiaoyu Chen, Yitong Yin, Xinyuan Zhang]{State Key Laboratory for Novel Software Technology, Nanjing University, 163 Xianlin Avenue, Nanjing, Jiangsu Province, China. \textnormal{E-mails: \url{chenxiaoyu233@smail.nju.edu.cn}, \url{yinyt@nju.edu.cn}, \url{zhangxy@smail.nju.edu.cn}}}
\address[Weiming Feng]{School of Informatics, University of Edinburgh, Informatics Forum, Edinburgh, EH8 9AB, United Kingdom. \textnormal{E-mail: \url{wfeng@ed.ac.uk}}}
\begin{document}

\begin{abstract}
  We prove an optimal $O(n \log n)$ mixing time of the Glauber dynamics for the Ising models with edge activity $\beta \in \left(\frac{\Delta-2}{\Delta}, \frac{\Delta}{\Delta-2}\right)$. This mixing time bound holds even if the maximum degree $\Delta$ is unbounded. 

We refine the boosting technique developed in~\cite{chen2021rapid}, and prove a new boosting theorem by utilizing the entropic independence defined in \cite{anari2021entropic}. The theorem relates the modified log-Sobolev (MLS) constant of the Glauber dynamics for a near-critical Ising model to that for an  Ising model in a sub-critical regime.

%We refine the boosting technique developed in~\cite{chen2021rapid}, and prove a boosting theorem for the entropic independence defined in \cite{anari2021entropic}, which relates the modified log-Sobolev (MLS) constant of the Glauber dynamics for a near-critical Ising model to that for an  Ising model in a sub-critical regime.
  
%  Our proof combines the entropic independence in~\cite{anari2021entropic} with the boosting technique in~\cite{chen2021rapid}.
%  We establish a new entropic boosting theorem, which relates the modified log-Sobolev (MLS) constant of the desired Glauber dynamics to the MLS constant of an easy-to-analyze Glauber dynamics.
  
  %Our result is based on a long line of researches discussing the connection between Glauber dynamics and high-dimensional expander.

  %echnique-wise, we build an entropic version of the boosting theorem in \cite{chen2021rapid}, which connects the modified log-sobolev constant of the desired Glauber dynamics and an easy-to-analyze Glauber dynamics. Our result is based on a long line of researches discussing the connection between Glauber dynamics and high-dimensional expander.
\end{abstract}

\maketitle

\section{Introduction}
% selected background of Ising model
The Ising model~\cite{ising1925beitrag} has been extensively studied in Statistic Physics, Probability and Computer Science.
Let $G=(V,E)$ be an undirected graph. 
Let $\beta \in \mathds{R}_{> 0}$ be the \emph{edge activity} and $\*\lambda = (\lambda_v)_{v \in V} \in \mathds{R}_{>0}^V$  the \emph{local fields}.
A  \emph{configuration} $\sigma \in \{\0, \1\}^V$ assigns each vertex $v\in V$ one of the two \emph{spins} in $\{\0,\1\}$.
%In Ising model, each vertex $v \in V$ can take one of two \emph{spins} in $\{\0,\1\}$.
%
The \emph{Gibbs distribution} $\mu$ over $\{\0, \1\}^V$ is defined by: %for any $\sigma \in \{\0, \1\}^V$,
\begin{align*}
\forall \sigma \in \{\0, \1\}^V,\quad \mu(\sigma) \triangleq \frac{1}{Z} \beta^{m(\sigma)} \prod_{v \in V: \sigma_v = \1} \lambda_v,
\end{align*}
where $m(\sigma) \triangleq \abs{\{\{u, v\} \in E \mid \sigma_u = \sigma_v\}}$ is the number of monochromatic edges in $\sigma$,
%$n_{+}(\sigma) \triangleq \abs{\{v \in V \mid \sigma_v = \1\}}$ denotes the number of vertices that take the value $+$ in $\sigma$, 
and the \emph{partition function} $Z=Z_G(\beta)$ is defined by
\begin{align*}
Z \triangleq \sum_{\sigma \in \{\0,\1\}^V}\beta^{m(\sigma)} \prod_{v \in V: \sigma_v = \1} \lambda_v. 	
\end{align*}
The Ising model is said to be \emph{ferromagnetic} if $\beta > 1$, and \emph{anti-ferromagnetic} if $\beta < 1$.
%If $\beta > 1$, the Ising model is ferromagnetic; if $\beta < 1$, it is anti-ferromagnetic. 

Sampling from Gibbs distributions is a fundamental computational task.
The \emph{Markov chain Monte Carlo} (MCMC) method is the most extensively studied technique for this task.
A canonical Markov chain is the \emph{Glauber dynamics} (a.k.a.~\emph{Gibbs sampler}, \emph{heat-bath}).
For any distribution $\mu$ over $\{\0,\1\}^V$, the chain starts from a \emph{feasible} configuration $X_0 \in\Omega(\mu) \subseteq\{\0,\1\}^V$.
In the $t$-th step, $X_{t-1}$ is updated to $X_t$ as:
\begin{itemize}
\item pick a vertex $v \in V$ uniformly at random and let $X_t(u) = X_{t-1}(u)$ for all $u\not= v$.
\item sample $X_t(v)$ from the marginal distribution $\mu_v^{X_{t-1}(V \setminus \{v\} )}$. 
\end{itemize}
Here, $\mu_v^{X_{t-1}(V \setminus \{v\} )}$ denotes the marginal distribution on $v$ projected from $\mu$ conditional on the values of all vertices except $v$ being fixed as $X_{t-1}(V \setminus \{v\} )$.
It is well known the chain is reversible with respect to $\mu$. 
Its convergence rate is captured by the \emph{mixing time}:
%We will use the \emph{mixing time} to measure the rate of the convergence:
\begin{align*}
  \forall 0 < \epsilon < 1, \quad T_{\mathrm{mix}}(\epsilon) &\triangleq \max_{X_0 \in \Omega(\mu)} \min \{t \mid \DTV{X_t}{\mu} \leq \epsilon\},
\end{align*}
where $\DTV{X_t}{\mu}$ is the \emph{total variational distance} between the distribution of $X_t$ and $\mu$.

The rapid mixing of the Glauber dynamics is intrinsically connected to the \emph{spatial mixing} property of the Gibbs distribution $\mu$.
For the Ising model, its spatial mixing is captured by the \emph{uniqueness condition}. 
Let $\Delta=\Delta_G$ denote the maximum degree of graph $G$.
There exist two critical thresholds  $\beta_c(\Delta)$ and $\bar{\beta}_c(\Delta)$ for the edge activity $\beta$, called the \emph{uniqueness thresholds}, defined by
\[
\beta_c(\Delta) = \frac{\Delta - 2}{\Delta} \quad\text{ and }\quad \bar{\beta}_c(\Delta) = \frac{\Delta}{\Delta - 2}.
\]
%$\beta_c(\Delta) = \frac{\Delta - 2}{\Delta}$ and $\bar{\beta}_c(\Delta) = \frac{\Delta}{\Delta - 2}$ for the edge activity $\beta$.
%
If $\beta_c(\Delta) < \beta < \bar{\beta}_c(\Delta)$, then for arbitrary local fields $\*\lambda \in \mathds{R}_{>0}^V$, the influence of a boundary condition on the marginal distribution at a vertex $v$ in $\mu$ decays exponentially in the distance between the vertex and the boundary;
and if $\beta <\beta_c(\Delta)$ or $\beta >  \bar{\beta}_c(\Delta)$,  then there exist such local fields that the boundary-to-vertex correlation persists as the distance between them grows to $\infty$~\cite{LLY13,sinclair2014approximation,guo2018uniqueness}.
%correlation between spins at vertices $u$ and $v$ decays  exponentially fast in the distance between $u$ and $v$.
%
%If $\beta <\beta_c(\Delta)$ or $\beta >  \bar{\beta}_c(\Delta)$, then for some local fileds,  the correlation persists even if the distance between $u$ and $v$ approaches to $\infty$.
%

%The Ising model gives a strong evidence to the above conjecture. 
The mixing behavior of Glauber dynamics for Ising model undergoes sharp transitions at the uniqueness thresholds.
On one hand, if $\beta$ lies in the non-uniqueness regime, the Glauber dynamics is torpid (slow) mixing.
For anti-ferromagnetic Ising model, in the non-uniqueness regime $\beta <\beta_c(\Delta)$, unless $\mathbf{NP}=\mathbf{RP}$, sampling from Ising models is not polynomial-time tractable~\cite{sly2012computational,galanis2016inapproximability}.
For ferromagnetic Ising model, in the non-uniqueness regime  $\beta > \bar{\beta}_c(\Delta)$,  the mixing time of Glauber dynamics is exponential in the size of the graph~\cite{GerschenfeldM07}, 
although the sampling problem can be solved otherwise by global Markov chains \cite{jerrum1993polynomial,guo2018random} or polynomial interpolation~\cite{LSS17}. %but in the non-uniqueness regime  $\beta > \bar{\beta}_c(\Delta)$, the mixing time of Glauber dynamics is exponential in the size of the graph~\cite{GerschenfeldM07}.

On the other hand, if $\beta$ lies in the uniqueness regime, the Glauber dynamics mixes in polynomial-time.
%
%On the other hand, for $\beta \in U$, the rapid mixing analysis of Glauber dynamics has been established.
In a seminal work of Mossel and Sly \cite{MS13}, 
by leveraging the monotonicity of coupling, 
an $\exp\tp{\Delta^{O(1/\delta)}} n\log n$ mixing time was proved for the ferromagnetic Ising models with $n$ vertices and $\beta \in (1, \frac{\Delta - \delta}{\Delta - 2 + \delta}]$, 
where $\delta \in (0, 1)$ controls the gap to the uniqueness threshold.
%
%they leveraged the monotonicity of the ferromagnetic Ising model and proved an $\exp\tp{\Delta^{O(1/\delta)}} n\log n$ mixing time for the Glauber dynamics when $\beta \in (1, \frac{\Delta - \delta}{\Delta - 2 + \delta}]$ using the coupling technique, where $\delta \in (0, 1)$ controls the gap between $\beta$ and $\bar{\beta}_c(\Delta)$ and $n$ is the number of vertices.
%
Recently, 
based on a powerful ``local-to-global'' theorem of Alev and Lau~\cite{anari2020spectral,alev2020improved} for high-dimensional expander walks,
an important concept called \emph{spectral independence} was introduced by Anari, Liu and Oveis Gharan in their seminal work~\cite{anari2020spectral},
which leads to a series of important progress in studies of mixing times~\cite{chen2020rapid,feng2021rapid,chen2021rapidcolor,liu2021coupling,blanca2021mixing,jain2021spectral,chen2021rapid,chen2021spectral,abdolazimi2021matrix, anari2021entropic}.
%introduced a powerful analyzing tool named \emph{spectral independence}~\cite{anari2020spectral,alev2020improved}, which leads to a flurry of impressive rapid mixing results~\cite{chen2020rapid,feng2021rapid,chen2021rapidcolor,liu2021coupling,blanca2021mixing,jain2021spectral,chen2021rapid,chen2021spectral,abdolazimi2021matrix}.
%
For the Ising models in the uniqueness regime with $\beta \in [\frac{\Delta - 2 + \delta}{\Delta - \delta}, \frac{\Delta - \delta}{\Delta - 2 + \delta}]$, %with uniform local fields (i.e. $\lambda_v = \lambda$ for all $v \in V$), many mixing results were established via spectral independence.
a polynomially-bounded mixing time $n^{O(1/\delta)}$ was first proved by Chen, Liu and Vigoda \cite{chen2020rapid} using the spectral independence;
and this mixing time bound was improved to $\Delta^{O(1/\delta)} n \log n$ in their subsequent work~\cite{chen2020optimal} by a uniform block factorization of entropy.
On general graphs with unbounded maximum degrees, both these mixing bounds become $n^{O(1/\delta)}$ in the worst case. 
%Although is mixing time is a polynomial, the parameter $\delta$ appears in the exponent of $n$.
%
Very recently, in a work~\cite{chen2021rapid} by the authors of the current paper, by establishing a boosting theorem for spectral gaps, an $\mathrm{e}^{O(1/\delta)}n^2$ mixing time bound was proved for the Ising models on general $n$-vertex graphs (with bounded or unbounded maximum degree $\Delta$) in the uniqueness regime with $\beta \in [\frac{\Delta - 2 + \delta}{\Delta - \delta}, \frac{\Delta - \delta}{\Delta - 2 + \delta}]$.
%
%The mixing time in~\cite{chen2020optimal} degenerates to $n^{O(1/\delta)}$ in graphs with large maximum degree $\Delta = n^{\Omega(1)}$.
%
%The result in~\cite{chen2021rapid} works for unbounded maximum degree but the dependency to $n$ is quadratic.
%

An important open problem is then to close the gap to the lower bound of the mixing time~\cite{hayes2007general,DY11} and prove an $O(n \log n)$ optimal mixing time for the Ising models in the uniqueness regime.
In this paper, we resolve the problem by proving the following theorem.

\begin{theorem}\label{thm:main}
For all $\delta \in(0,1)$, there exists $C_\delta=\exp(O(1/\delta))$ such that for every Ising model on $n$-vertex graph $G=(V,E)$ with maximum degree $\Delta=\Delta_G\ge 3$  and with edge activity  $\beta\in\left[\frac{\Delta-2+\delta}{\Delta-\delta},\frac{\Delta-\delta}{\Delta-2+\delta}\right]$ and local fields $\*\lambda \in \mathds{R}_{>0}^V$, 
the mixing time of the Glauber dynamics is bounded as
\begin{align*}
\tmix(\epsilon) \leq C_\delta \frac{\lambda_{\max}}{\lambda_{\min}} n \tp{ \log \frac{n}{\epsilon} + \log \log \frac{2\lambda_{\max}}{\lambda_{\min}}},
\end{align*}
where $\lambda_{\max} =  \max_{v \in V} \lambda_v$ and $\lambda_{\min}=\min_{v \in V} \lambda_v$.
\end{theorem}

The theorem gives an optimal $O(n\log n)$ mixing time for both ferromagnetic and anti-ferromagnetic Ising models on general graphs (with bounded or unbounded maximum degrees) in the interior of the uniqueness regime, for constant local fields or the local fields that are in the same order of magnitude. 
%
%, which improves~\cite{MS13,chen2020rapid,chen2020optimal,chen2021rapid}.
%so that it extends the result in~\cite{MS13}.
%Moreover, our result achieves the optimal $O_{\delta, \*\lambda}(n \log n)$ mixing time even if the maximum degree of the graph is unbounded, which improves upon the recent mixing time results in~\cite{chen2020optimal,chen2021rapid}.
%For the Ising model with uniform local fields (i.e. $\lambda_v = \lambda$ for all $v \in V$), our result achieves the optimal $O(n \log n)$ mixing time, even if 
%The mixing time in \Cref{thm:main} is optimal if the ratio $\frac{\lambda_{\max}}{\lambda_{\min}}$ is a constant, which covers the case of .  
%

%
We further remark that by combining techniques in this paper with a Markov chain called the field dynamics defined in \cite{chen2021rapid}, we can obtain an $O_{\delta}(n \log^2 n)$-time sampling algorithm for the Ising models with $\beta\in\left[\frac{\Delta-2+\delta}{\Delta-\delta},\frac{\Delta-\delta}{\Delta-2+\delta}\right]$, where the constant in $O_{\delta}(\cdot)$ depends only on $\delta$  but not on $\*\lambda$. 
This sampling algorithm is postponed to the full version of the paper.

\subsection{Results for general distributions}
We prove a boosting theorem for the \emph{modified log-Sobolev (MLS)} constant for Glauber dynamics on general distributions. 
Bounds on MLS constants  can often give optimal mixing time for Glauber dynamics.
Let $V$ be a finite ground set.
Let $\mu$ be a distribution over $\{\0,\1\}^V$.
We use $\Omega(\mu)$ to denote the support of distribution $\mu$.
Let $P = \PGD_\mu :\Omega(\mu) \times \Omega(\mu) \to \mathds{R}_{\geq 0}$ denote the transition matrix of the Glauber dynamics on $\mu$.
For any function $f: \Omega(\mu)\to \mathds{R}_{\geq 0}$, its \emph{Dirichlet form} is defined by:
\begin{align*}
\+E_{P}(f, \log f) \triangleq \inner{f}{(I-P)\log f}_{\mu}, %\triangleq \sum_{\sigma \in \Omega(\mu)}\mu(\sigma)f(\sigma)(P\log f)(\sigma).	
\end{align*}
where the inner product $\inner{f}{g}_{\mu}\triangleq \sum_{\sigma \in \Omega(\mu)}f(\sigma)g(\sigma)\mu(\sigma)$.
Define  the entropy:
\begin{align*}
\Ent[\mu]{f} \triangleq \EE[\mu]{f \log f} - \EE[\mu]{f} \log \EE[\mu]{f},	
\end{align*}
where $\EE[\mu]{f} \triangleq \sum_{\sigma \in \Omega(\mu)}\mu(\sigma)f(\sigma)$.
We use the convention $0\log 0 = 0$ in above definitions.

Our goal is to bound the following  \emph{modified log-Sobolev constant}~\cite{bobkov2006modified} for the Glauber dynamics on $\mu$:
\begin{align} \label{eq:mlsc} % modified log Sobolev constant
  \rho^{\mathrm{GD}}(\mu) &\triangleq \inf \left\{ \left.\frac{\+E_{P}(f, \log f)}{\Ent[\mu]{f}} \;\right\vert\; f:\Omega(\mu) \to \mathds{R}_{\geq 0},\; \Ent[\mu]{f} \not= 0 \right\},
\end{align}
so that the mixing time of Glauber dynamics can be bounded by,
\begin{align*}
  t_{\mathrm{mix}}(\epsilon) \leq \frac{1}{\rho^{\mathrm{GD}}(\mu)} \tp{\log \log \frac{1}{\mu_{\min}} + \log \frac{1}{2\epsilon^2}},
\end{align*}
where $\mu_{\min} \triangleq \min_{\sigma \in \Omega(\mu)} \mu(\sigma)$ is the minimum probability in $\mu$.

%To state our main result, we first introduce some definitions.
%
For any distribution $\mu$ over $\{\0,\1\}^V$, any subset $S \subseteq V$, let $\mu_S$ denote the marginal distribution on $S$ induced by $\mu$.
We simply use $\mu_v$ to denote $\mu_{\{v\}}$ if $S = \{v\}$.
For any subset $\Lambda \subseteq V$, any $\sigma \in \Omega(\mu_\Lambda)$, we use $\mu^{\Lambda \gets \sigma}$ (or simply $\mu^\sigma$ if $\Lambda$ is clear from the context) to denote the distribution over $\{\0,\1\}^V$ induced by $\mu$ conditional on the configuration $\sigma$ on $\Lambda$.
%%
%The \emph{signed influence matrix} for $\mu$ is defined as follows.
\begin{definition}[\text{signed influence matrix~\cite{anari2020spectral}}]
\label{definition-inf-matrix}
Let $\mu$ be a distribution over $\{\0,\1\}^V$.
The \emph{signed influence matrix} $\Ima_{\mu}: V\times V \to \mathds{R}$ is defined by
\begin{align*}
\forall u,v \in V, \quad \Ima_{\mu}(u,v) = \begin{cases}
    \mu_v^{u \gets \1}(\1) - \mu_v^{u \gets \0}(\1) &\text{if } u \neq v \text{ and } \Omega(\mu_u)=\{\0,\1\};\\
 	0 &\text{otherwise}.
 \end{cases}
\end{align*}	
\end{definition}

%We say a distribution $\mu$ over $\{\0,\1\}^V$ has \emph{bounded total influence with arbitrary local fields} if the following property is satisfied.
In \cite{anari2020spectral}, a notion of \emph{spectral independence} was defined using the signed influence matrix.  
Specifically, a distribution $\mu$ over $\{-1,+1\}^V$ is spectrally independent if its influence matrix $\Ima_{\mu}$ has bounded spectral radius.
In this paper, we use the following sufficient condition for the spectral independence (SI).

\begin{definition}[spectral independence in ${\infty}$-norm]
Let $\eta>0$.
A distribution $\mu$ over $\{\0,\1\}^V$ is said to be \emph{$\eta$-spectrally independent in ${\infty}$-norm} if 
%have $\eta$-\emph{bounded total influence with arbitrary local fields} if for any  positive local fields $\*\lambda = (\lambda_v)_{v \in V} \in \mathds{R}_{>0}^V$,
\begin{align*}
\norm{\Ima_{\mu}}_\infty = \max_{u \in V}\sum_{v \in V}\abs{\Ima_{\mu}(u,v)} \leq \eta.
%\norm{\Ima_{\mu^{(\*\lambda)}}}_\infty = \max_{u \in V}\sum_{v \in V}\abs{\Ima_{\mu^{(\*\lambda)}}(u,v)} \leq \eta.
\end{align*}
%$\lambda_{\max}(\Psi^{\mathrm{inf})$		
\end{definition}

Next, we introduce a notation for the distributions with local fields. 
\begin{definition}[magnetizing a joint distribution with local fields]
\label{definition-local-fields}
Let $\mu$ be a distribution over $\{\0,\1\}^V$.
Let $\*\lambda = (\lambda_v)_{v \in V} \in \mathds{R}_{>0}$ be positive local fields.
Let $\mu^{(\*\lambda)}$ denote the distribution obtained from imposing the local fields $\*\lambda$ onto $\mu$. Formally, for any configuration $\sigma \in \{\0,\1\}^V$,
\begin{align*}
\mu^{(\*\lambda)}(\sigma) \propto \mu(\sigma) \prod_{v \in V: \sigma_v = \1}\lambda_v.	
\end{align*}
We simply denote $\mu^{(\*\lambda)}$ by $\mu^{(\lambda)}$ if $\lambda_v = \lambda$ for all $v \in V$.
\end{definition}

\begin{definition}[spectral independence in $\infty$-norm with all fields]
%Let $\eta>0$.
A distribution $\mu$ over $\{\0,\1\}^V$ is said to be $\eta$-spectrally independent in ${\infty}$-norm \emph{with all fields} if 
${\mu^{(\*\lambda)}}$ is $\eta$-spectrally independent in $\infty$-norm for all $\*\lambda = (\lambda_v)_{v \in V} \in \mathds{R}_{>0}^V$.
\end{definition}

We define the MLS constant under the worst pinning. 
Let $\pi$ be a distribution over $\{\0,\1\}^V$.
For any $\Lambda \subseteq V$, any $\sigma \in \Omega(\pi_\Lambda)$, let $\rho^{\mathrm{GD}}(\pi^\sigma)$ denote the MLS constant for the Glauber dynamics on $\pi^\sigma$. Let
\begin{align}
\label{eq-def-MLS-min}
\rho^{\mathrm{GD}}_{\min}(\pi) \triangleq \min_{\Lambda \subset V}\min_{\sigma \in \Omega(\pi_\Lambda)} \rho^{\mathrm{GD}}(\pi^\sigma).	
\end{align}
%We have the following main result.

We now state our  boosting theorem for the MLS constant for general distributions over $\{\0,\1\}^V$.
\begin{theorem}[boosting theorem for modified log-Sobolev]
\label{theorem-general}
Let $\mu$ be a distribution over $\{\0,\1\}^V$ and  $\eta>0$.
%If ${\mu^{(\*\lambda)}}$ is $\eta$-spectrally independent in $\infty$-norm for all $\*\lambda = (\lambda_v)_{v \in V} \in \mathds{R}_{>0}^V$,
If ${\mu}$ is $\eta$-spectrally independent in $\infty$-norm with all fields,
then the following holds for the modified log-Sobolev constants for Glauber dynamics: %for any $\theta \in (0,1)$,
\begin{align*}
\forall\theta\in(0,1),\quad \rho^{\mathrm{GD}}(\mu) \geq \tp{\frac{\theta}{\e}}^{\eta + 3}\rho^{\mathrm{GD}}_{\min}(\mu^{(\theta)}).
\end{align*}
\end{theorem}

%\begin{remark}
%\Cref{theorem-general} requires that $\mu$ has bounded total influence with arbitrary local fields. We actually prove the theorem under a slightly weaker condition. See \Cref{corollary-muk-SI} for details.
%\end{remark}

\Cref{theorem-general} is a boosting theorem for the MLS constant.
When applied on a distribution $\mu$ in a near-critical regime,
by carefully choosing a parameter $\theta$ so that $\mu^{(\theta)}$ enters an easier sub-critical regime,
the MLS constant bound for $\mu$ is reduced to easier or known MLS constant bounds,
by losing a  $\theta^{O(\eta)}$ factor.
%
%Hence, We can analyze $\rho^{\mathrm{GD}}_{\min}(\mu^{(\theta)})$ by existing tools, then apply the boosting theorem to bound the MLS constant for distribution $\mu$, and such a boosting only loses a factor of $\theta^{O(\eta)}$.
%
%The booting lemma introduces an extra factor $(\frac{\theta}{2})^{2\eta + 7}$, which is a constant if $\theta = \Omega(1)$ and $\eta = O(1)$.

%
A similar boosting theorem for the Poincar\'{e} constant (spectral gap) was established in~\cite{chen2021rapid}, which requires $\mu^{(\*\lambda)}$ to be spectrally independent for all $\*\lambda \in (0,1]^V$.
\Cref{theorem-general} asks for a stronger condition,  but holds for the MLS constant that can imply optimal mixing time bounds.
%it works for the MLS constant, with which we can prove the optimal mixing time result.

%where relates the modified log-Sobolev constant for a critical distribution $\mu$ to that for a non-critical distribution $\mu^{(\theta)}$.  

\subsection{Open problems}
The key ingredient in this paper is the boosting theorem for MLS constant, which connects the MLS constants of Glauber dynamics on distributions $\mu$ and $\mu^{(\theta)}$. 
However, the distribution $\mu$ should be spectrally independent in $\infty$-norm with \emph{arbitrary} local fields, which restricts its applicability. 
An open problem is to prove a boosting theorem with weaker assumption, especially the spectral independence with one-sided local fields,
so that it can be used to prove optimal mixing times for other 2-spin systems, e.g.~the hard-core model. 
%For example, if the distribution $\mu$ satisfies that $\mu^{(\*\lambda)}$ has bounded total influence for all $\*\lambda \in (0,1]^V$, can we still obtain a similar boosting theorem?
%For example, can we still obtain a similar boosting lemma if $\mu$ is completely spectral independent\cite{chen2021rapid}, i.e. $\mu$ is $\eta$-spectrally independent with local fields that is bounded by $1$?

Another important question is to extend our framework to distributions beyond the boolean domain so that it can be used to prove optimal mixing times for general spin systems, e.g.~proper $q$-colorings of general graphs with bounded or unbounded degrees.

Finally, for the Ising model, our bound for mixing time of Glauber dynamics depends on the ratio $\frac{\lambda_{\max}}{\lambda_{\min}}$. 
This ratio comes from the analysis of the MLS constant in the easier regime $\mu^{(\theta)}$ (see \Cref{section-Ising} for details).
We conjecture that such dependency on the local fields in the mixing time can be removed.
This requires a deeper understanding of the MLS constant to the Ising models with heterogenous (and non-constant) fields, 
which has not been very well studied even in an easy regime.
%
%Can we give a better analysis of $\rho^{\mathrm{GD}}_{\min}(\mu^{(\theta)})$ to remove the dependence on local fields? 
%Can we further optimize to remove the dependence of local fields? 

\section{Proof outline}
We outline the proof of the boosting theorem (\Cref{theorem-general}).
The result for Ising model (\Cref{thm:main}) is an application of the boosting theorem.
The proof of \Cref{thm:main} is given in \Cref{section-Ising}.

To prove \Cref{theorem-general}, the key step is to establish the ``{magnetized block factorization of entropy}''. Let $\theta \in (0,1)$ and $\pi = \mu^{(\theta)}$. %
Formally, for every $\sigma \in \{\0,\1\}^V$,
\begin{align}
\label{eq-redef-pi}
\pi(\sigma) = \frac{\mu(\sigma)\theta^{\|\sigma\|_{+}} }{Z_\pi},
\quad\text{ where }Z_\pi \triangleq \sum_{\sigma \in \{\0,\1\}^V}\mu(\sigma) \theta^{\|\sigma\|_{+}} \text{ and } \|\sigma\|_{+} = |\{v \in V \mid \sigma_v = \1\}|.
\end{align}

We use $\Bin(V,1-\theta)$ to denote the distribution of random subset $R \subseteq V$ generated by including each $v \in V$ into $R$ independently with probability $1 - \theta$.
Specifically, for every $\Lambda \subseteq V$,
\begin{align*}
  \Pr[R \sim \Bin(V,1-\theta)]{R = \Lambda} = (1 - \theta)^{\abs{\Lambda}}\theta^{\abs{V} - \abs{\Lambda}}.
\end{align*}

\begin{definition}[magnetized block factorization of entropy]
\label{definition-key-BF}
Let $\theta \in (0,1)$ and $C > 0$.
%Let $\mu$ be a distribution over $\{\0,\1\}^V$, where $n = |V|$.
%Let $\pi = \mu^{(\theta)}$.
%Let $V$ be a ground set of size $n=|V|$.
A distribution $\mu$ over $\{\0,\1\}^V$
is said to satisfy \emph{$\theta$-magnetized block factorization of entropy with parameter $C$} 
if the following holds for $\pi = \mu^{(\theta)}$ and for all $f: \Omega(\mu) \to \mathds{R}_{\geq 0}$:
\begin{align}
  \label{eq-lem-FBF}
    \Ent[\mu]{f} \leq C \cdot\frac{Z_\pi}{\theta^{|V|}} \E[R \sim \Bin(V,1-\theta)]{\pi_{R}\tp{\mathds{1}_{R}} \cdot \Ent[\pi^{\mathds{1}_{R}}]{f}},
  \end{align}
assuming $\pi_R\tp{\mathds{1}_R} \cdot \Ent[\pi^{\mathds{1}_R}]{f} = 0$ if $\pi_R\tp{\mathds{1}_R} = 0$,
where $\mathds{1}_R$ denotes the all-(+1) configuration on $R\subseteq V$.
\end{definition}

In the left-hand-side (LHS) of~\eqref{eq-lem-FBF}, the entropy is evaluated with respect to the original distribution $\mu$; and in the right-hand-side (RHS), the entropy is evaluated with respect to the  conditional distribution $\pi^{\mathds{1}_R}$, where $\pi = \mu^{(\theta)}$ is obtained from $\mu$ by changing the local fields by a factor $\theta$. 
This definition plays a key role to relate $\mu$ to $\mu^{(\theta)}$.
\Cref{theorem-general} can be proved by combining the following two lemmas.
\begin{lemma}\label{lemma-imply-mls}
Let $\theta \in (0,1)$ and $C > 0$. For any distribution $\mu$ over $\{\0,\1\}^V$ satisfying $\theta$-magnetized block factorization of entropy with parameter $C$, the Glauber dynamics on $\mu$ has the modified log-Sobolev constant
\begin{align*}
\rho^{\mathrm{GD}}(\mu) \geq \frac{\rho^{\mathrm{GD}}_{\min}(\mu^{(\theta)})}{C}.	
\end{align*}
\end{lemma}

\begin{lemma} \label{lem:key} % field block factorization
Let $\mu$ be a distribution over $\{\0,\1\}^V$ and $\eta > 0$.
If  ${\mu}$ is $\eta$-spectrally independent in $\infty$-norm with all fields,
%${\mu^{(\*\lambda)}}$ is $\eta$-spectrally independent in $\infty$-norm for all $\*\lambda = (\lambda_v)_{v \in V} \in \mathds{R}_{>0}^V$,
then $\mu$ satisfies $\theta$-magnetized block factorization of entropy with parameter 
  \begin{align*}
  	C = \tp{\frac{\e}{\theta}}^{\eta + 3}.
  \end{align*}
  %then for any $f: \{\0, \1\}^V \to \mathds{R}_{\geq 0}$, it holds that
\end{lemma}

The proof of \Cref{lemma-imply-mls} is outlined in \Cref{section-imply-mls}, and the proof of \Cref{lem:key} is outline in \Cref{section-lem-key}.

\subsection{The MLS constant bounds}
\label{section-imply-mls}
%\Cref{lem:FBF} gives an upper bound of $\Ent[\mu]{f}$.
To obtain the MLS constant for the Glauber dynamics $P = \PGD_\mu$ on $\mu$, we need to relate the RHS of~\eqref{eq-lem-FBF} with the Dirichlet form $\+E_{P}(f,\log f)$. Formally, we prove the following inequality for any function $f: \{\0, \1\}^V \to \mathds{R}_{\geq 0}$:
\begin{align} \label{eq:compare-target-ouline}
\frac{Z_\pi}{\theta^{n}} \E[R \sim \Bin(V,1-\theta)]{\pi_{R}\tp{\mathds{1}_{R}} \cdot \Ent[\pi^{\mathds{1}_{R}}]{f}} \leq \frac{\+E_{P}(f, \log f)}{\rho^{\mathrm{GD}}_{\min}(\pi)}.
 \end{align}
Combining~\eqref{eq-lem-FBF} and~\eqref{eq:compare-target-ouline} implies $\Ent[\mu]{f} \leq \frac{C}{\rho^{\mathrm{GD}}_{\min}(\pi)}\+E_{P}(f, \log f)$. This proves~\Cref{lemma-imply-mls}.

%\begin{lemma} \label{lem:compare}
%Let $\theta \in (0, 1)$ and $\mu$ be a distribution over $\{\0, \1\}^V$. %and $\theta \in (0, 1)$ be a real number.
%  For any function $f: \{\0, \1\}^V \to \mathds{R}_{\geq 0}$,
%  \begin{align} \label{eq:compare-target}
%  \rho^{\mathrm{GD}}_{\min}(\pi) \cdot\frac{Z_\pi}{\theta^{n}} \E[R \sim \Bin(V,1-\theta)]{\pi_{R}\tp{\mathds{1}_{R}} \cdot \Ent[\pi^{\mathds{1}_{R}}]{f}} \leq \+E_{P}(f, \log f),
%  \end{align}
%  where $n = |V|$, $\pi = \mu^{\tp{\*\theta}}$ and $ \rho^{\mathrm{GD}}_{\min}(\pi)$ is defined in~\eqref{eq-def-MLS-min}.
%\end{lemma}

Inequality~\eqref{eq:compare-target-ouline} can be verified by a careful comparison between the MLS constants of the original Glauber dynamics on $\mu$ and the Glauber dynamics on $\pi^{\mathds{1}_{R}}$. 
For each $\pi^{\mathds{1}_{R}}$ , %let $P(\pi^{\mathds{1}_{R}})$ denote the transition matrix of Glauber dynamics on $\pi^{\mathds{1}_{R}}$. 
the entropy $\Ent[\pi^{\mathds{1}_{R}}]{f}$ can be bounded by the modified log-Sobolev inequality for Glauber dynamics $P(\pi^{\mathds{1}_{R}})$ on $\pi^{\mathds{1}_{R}}$:
\begin{align*}
 \Ent[\pi^{\mathds{1}_{R}}]{f} \leq \frac{\+E_{P(\pi^{\mathds{1}_{R}})}(f, \log f)}{ \rho^{\mathrm{GD}}(\pi^{\mathds{1}_{R}})} \leq  \frac{\+E_{P(\pi^{\mathds{1}_{R}})}(f, \log f)}{ \rho^{\mathrm{GD}}_{\min}(\pi)},
\end{align*}
where $\+E_{P(\pi^{\mathds{1}_{R}})}(f, \log f)$ is the Dirichlet form  for the chain $P(\pi^{\mathds{1}_{R}})$.
Applying this inequality in the LHS of~\eqref{eq:compare-target-ouline}, 
an upper bound ${\+E_{P}(f, \log f)}$ with the original chain $P$ emerges  from taking  an ``average'' over all Dirichlet forms $\+E_{P(\pi^{\mathds{1}_{R}})}(f, \log f)$.
%we can take an ``average'' over all Dirichlet forms $\+E_{P(\pi^{\mathds{1}_{R}})}(f, \log f)$, then obtain the desired upper bound $\frac{\+E_{P}(f, \log f)}{\rho^{\mathrm{GD}}_{\min}(\pi)}$.
%
The detailed calculation is given in \Cref{section:compare-FD-GD}.

%\begin{proof}[Proof of \Cref{theorem-general}]
%....	
%\end{proof}

\subsection{Magnetized block factorization of entropy}
\label{section-lem-key}
We outline the proof of \Cref{lem:key}.
To establish the magnetized block factorization of entropy, %we use a limit analysis. 
we use the following transformation of distributions. 
\begin{definition}[\text{$k$-transformation~\cite{chen2021rapid}}] \label{def:k-trans}
Let $\mu$ be a distribution over $\{\0, \1\}^V$ and $k \geq 1$ an integer.
The \emph{$k$-transformation} of $\mu$, denoted by $\mu_k = \Rd(\mu, k)$, is a distribution over $\{\0,\1\}^{V\times[k]}$ defined as follows.
Let $\*X\sim\mu$. Then $\mu_k = \Rd(\mu, k)$ is the distribution of $\*Y\in\{\0,\1\}^{V\times [k]}$ constructed as follows:
\begin{itemize}
  \item if $X_v = \0$, then $Y_{(v,i)} = \0$ for all $i\in[k]$;
  \item if $X_v = \1$, then $Y_{(v,i^*)} = \1$ and $Y_{(v,i)} = \0$ for all $i\in[k]\setminus \{i^*\}$, where $i^*$ is chosen from $[k]$ uniformly and independently at random.
\end{itemize}
\end{definition}

%Our proof is to analyze all distributions $\mu_k$ for all $k \geq 1$.
%We next introduce the following notion of spectral independence with arbitrary local fields. 
%
%\begin{definition}[spectral independence with arbitrary local fields]
%\label{definition-SI-local}
%Let $\eta >0$.
%A distribution $\mu$ over $\{\0,\1\}^V$ is said to be $\eta$-\emph{spectrally independent with arbitrary local fields} if 
%$\lambda_{\max} (\Ima_{\mu^{(\*\lambda)}}) \leq \eta$ for any  positive local fields $\*\lambda = (\lambda_v)_{v \in V} \in \mathds{R}_{>0}^V$.
%%$\lambda_{\max}(\Psi^{\mathrm{inf})$	
%\end{definition}

%It is straightforward to verify that is $\mu$ has $\eta$-bounded total influence with arbitrary local fields, then $\mu$ is $\eta$-spectrally independent with arbitrary local fields.
\Cref{lem:key} is proved by combining the following two lemmas that are claimed for arbitrary distribution $\mu$ over $\{\0, \1\}^V$ and  $\eta > 0$.

\begin{lemma} \label{lem:muk-inf}
  If  ${\mu}$ is $\eta$-spectrally independent in $\infty$-norm with all fields,
%  ${\mu^{(\*\lambda)}}$ is $\eta$-spectrally independent in $\infty$-norm for all $\*\lambda = (\lambda_v)_{v \in V} \in \mathds{R}_{>0}^V$,
%  If $\mu$ has $\eta$-bounded total influence with arbitrary local fields, 
  then for every integer $k \geq 1$, the $k$-transformed distribution 
   ${\mu_k}= \Rd(\mu, k)$ is $(\eta+1)$-spectrally independent in $\infty$-norm with all fields.
%  ${\mu_k^{(\*\lambda)}}$  is $(\eta+1)$-spectrally independent in $\infty$-norm for all $\*\lambda = (\lambda_v)_{v \in V} \in \mathds{R}_{>0}^V$.
  %has $(\eta+1)$-bounded total influence with arbitrary local fields. %which implies $\mu_k$ is $(\eta+1)$-spectrally independent with arbitrary local fields.
\end{lemma}

\begin{lemma}\label{lem:muk-main}
% Let $\mu$ be a distribution over $\{\0, \1\}^V$ and $\eta > 0$.
If for every integer $k \geq 1$, 
 ${\mu_k}= \Rd(\mu, k)$ is $\eta$-spectrally independent in $\infty$-norm with all fields,
%${\mu_k^{(\*\lambda)}}$ is $\eta$-spectrally independent in $\infty$-norm for all $\*\lambda = (\lambda_v)_{v \in V} \in \mathds{R}_{>0}^V$,
%$\mu_k$ has $\eta$--bounded total influence with arbitrary local fields, 
then for every $\theta \in (0,1)$,  $\mu$ satisfies $\theta$-magnetized block factorization of entropy with parameter $C = \tp{\frac{\e}{\theta}}^{\eta + 2}$.
\end{lemma}

%Besides, according to \Cref{lem:muk-main}, the result in \Cref{theorem-general} holds with a weaker condition.
%\begin{corollary}
%\label{corollary-muk-SI}	
%The result in \Cref{theorem-general} holds if $\mu_k$ is $\eta$-spectrally independent with arbitrary local fields for any integer $k \geq 1$.
%\end{corollary}

\Cref{lem:muk-inf} is proved by a coupling between $\mu$ and $\mu_k$. 
The proof is given in \Cref{section-SI-muk}.

We now give an outline of the proof of \Cref{lem:muk-main}.
First, we show that for all sufficiently large $k$, the distribution $\mu_k$ satisfies the \emph{uniform block factorization of entropy} in~\cite{chen2020optimal}.
Next, we show that when $k \to \infty$, the uniform block factorization of entropy for $\mu_k$ implies the magnetized block factorization of entropy  for the original distribution $\mu$.

%The uniform block factorization of the entropy is defined as follows. %For any integer $1\leq \ell \leq |V|$, let $\binom{V}{\ell}$ denote all subset of $V$ with size $\ell$.
\begin{definition}[uniform block factorization of entropy \text{\cite{CP20}}]
%Let $\mu: \{\0,\1\}^{V} \to \mathds{R}_{\geq 0}$ be a distribution with support $\Omega(\mu)$, where $n = |V|$.
Let $V$ be a set of size $n=|V|$, $1 \leq \ell \leq n$ an integer, and $C>0$.
A distribution $\mu$ over $\{\0,\1\}^{V}$ is said to satisfy the \emph{$\ell$-uniform block factorization of entropy with parameter $C$} if for all $f:\Omega(\mu) \to \mathds{R}_{\geq 0}$,
\begin{align*}
\Ent[\mu]{f} \leq \frac{C}{\binom{n}{\ell}} \sum_{S \in \binom{V}{\ell}}\mu[\Ent[S]{f}],
\end{align*}
where $\mu[\Ent[S]{f}] \triangleq \sum_{\sigma \in \Omega(\mu_{V \setminus S})}\mu_{V \setminus S}(\sigma)\cdot\Ent[\mu^\sigma]{f}$. %and $\binom{V}{\ell}$ denotes all subset of $V$ with size $\ell$.
\end{definition}

The following lemma shows that for all sufficiently large integers $k$, the transformed distribution $\mu_k$ satisfies the uniform block factorization of entropy.
The lemma can be proved using the approach of \emph{entropic independence}  developed in~\cite{AASV21,anari2021entropic}. A formal proof is included in \Cref{section-proof-AJK+}.
\begin{lemma}[\text{\cite{AASV21,anari2021entropic}}]
\label{corollary-muk-UBF}
Let $\mu$ be a distribution over $\{\0,\1\}^V$ with $n=|V|$.
Let $\eta > 0$ and $\theta \in (0,1)$.
%Let $\eta > 0$, $\theta \in (0,1)$ and $V$ be a set of size $n$.
For all integers $k > \frac{\eta + 2}{\theta n}$, for the $k$-transformed distribution $\mu_k = \Rd(\mu, k)$,
if  ${\mu_k}$ is $\eta$-spectrally independent in $\infty$-norm with all fields,
%${\mu_k^{(\*\lambda)}}$ is $\eta$-spectrally independent in $\infty$-norm for all $\*\lambda = (\lambda_v)_{v \in V} \in \mathds{R}_{>0}^V$,
%If $\mu_k$ has $\eta$-bounded total influence with arbitrary local fields, 
then $\mu_k$ satisfies $\ctp{\theta k n}$-uniform block factorization of entropy with $C = \tp{\frac{\mathrm{e}}{\theta}}^{\eta + 2}$.
\end{lemma}
%The proof of \Cref{corollary-muk-UBF} is given in \Cref{section-proof-AJK+}.

Our next lemma relates the uniform block factorization of entropy for the $k$-transformed distribution $\mu_k = \Rd(\mu, k)$ to  the magnetized block factorization of entropy for $\mu$. 
A formal proof is given in \Cref{sec:FBF}.
%Formally, we prove the following lemma.

\begin{lemma} \label{lem:FBF} % field block factorization
Let $\mu$ a distribution over $\{\0,\1\}^V$ with $n=|V|$. Let $\theta \in (0, 1)$ and $C > 0$. 
  If there is a finite $k_0$ such that for all integers $k\ge k_0$, the distribution $\mu_k = \Rd(\mu, k)$ satisfies $\lceil \theta kn \rceil$-uniform block factorization with parameter $C$, then $\mu$ satisfies $\theta$-magnetized block factorization of entropy with parameter $C$.
\end{lemma}

\Cref{lem:muk-main} is a straightforward consequence of \Cref{corollary-muk-UBF} and \Cref{lem:FBF}.

Finally, we briefly outline the proof of \Cref{lem:FBF}.
Fix an integer $k \ge k_0$. We use $V_k$ to denote the  $V \times [k]$ and $v_i$ to denote the pair $(v,i)$. 
Since $\mu_k$ satisfies the $\lceil \theta kn \rceil$-uniform block factorization with parameter $C$, %under the condition in \Cref{lem:muk-main}, 
it holds that for any $g: \Omega(\mu_k) \to \mathds{R}$,
\begin{align}
\label{eq-muk-ubf}
\Ent[\mu_k]{g} \leq \frac{C}{\binom{nk}{\ctp{\theta n k}}} \sum_{S \in \binom{V_k}{\ctp{\theta n k}}}\mu_k\left[\Ent[S]{g}\right] = C \cdot \EE[S \sim \binom{V_k}{\ctp{\theta n k}}]{\mu_k\left[\Ent[S]{g}\right]},	
\end{align}
where $S \sim \binom{V_k}{\ctp{\theta n k}}$ is sampled uniformly at random from $\binom{V_k}{\ctp{\theta n k}}$.
%\subsubsection{Step-2: Map transformed distributions to the original distribution}
%Next, we show that uniform block factorization of transformed distributions implies a good property of the original distribution.
%Let $\theta \in (0,1)$ and $\pi = \mu^{(\theta)}$. %
%Formally, for every $\sigma \in \{\0,\1\}^V$,
%\begin{align}
%\label{eq-redef-pi}
%\pi(\sigma) = \frac{\mu(\sigma)\theta^{\|\sigma\|_{+}} }{Z_\pi},
%\quad\text{ where }Z_\pi \triangleq \sum_{\sigma \in \{\0,\1\}^V}\mu(\sigma) \theta^{\|\sigma\|_{+}} \text{ and } \|\sigma\|_{+} = |\{v \in V \mid \sigma_v = \1\}|.
%\end{align}
%
%We may use $\Pr[R\subseteq V]{\cdot}$ to denote the law for subset $R \subseteq V$ that is randomly generated by by including each $v \in V$ into $R$ independently with probability $1 - \theta$.
%Specifically, for every $\Lambda \subseteq V$,
%\begin{align*}
%  \Pr[R \subseteq V]{R = \Lambda} = (1 - \theta)^{\abs{\Lambda}}\theta^{\abs{V} - \abs{\Lambda}}.
%\end{align*}
By the definition of the $k$-transformation, a mapping from $\Omega(\mu_k)$ to $\Omega(\mu)$ can be naturally constructed as follows: 
For any $\sigma \in \Omega(\mu_k)$, %define
\begin{align*}
\forall v \in V,\quad \sigma^\star_v = \begin{cases}
 \1 & \exists i \in [k] \text{ s.t. } \sigma_{(v,i)} = \1;\\
 \0 & \forall i \in [k], \sigma_{(v,i)} = \0.
 \end{cases}
\end{align*}
Given any function $f: \Omega(\mu) \to \mathds{R}_{\geq 0}$, define a function $f^k: \Omega(\mu_k) \to \mathds{R}_{\geq 0}$ by $f^k(\sigma) = f(\sigma^\star)$ for all $\sigma \in \Omega(\mu_k)$. The following identities can be verified:
\begin{align}
\Ent[\mu]{f} &= \Ent[\mu_k]{f^k}\label{eq-rela-1},\\
\frac{Z_\pi}{\theta^{n}} \E[R \sim \Bin(V,1-\theta)]{\pi_{R}\tp{\mathds{1}_{R}} \cdot \Ent[\pi^{\mathds{1}_{R}}]{f}} &= \lim_{k \to \infty}\EE[S \sim \binom{V_k}{\ctp{\theta n k}}]{\mu_k\left[\Ent[S]{f^k}\right]}.\label{eq-rela-2}
\end{align}
\Cref{eq-rela-1} follows from the bijection. 
\Cref{eq-rela-2} follows from the concentration property of the random set $S$. 
Applying \eqref{eq-rela-1} and~\eqref{eq-rela-2} together with~\eqref{eq-muk-ubf} and letting $k \to \infty$ gives:
\begin{align*}
\Ent[\mu]{f} &\overset{\text{by~\eqref{eq-rela-1}}}{=}  \lim_{k \to \infty} \Ent[\mu_k]{f^k}\\
&\overset{\text{by}~\eqref{eq-muk-ubf}}{\leq} \lim_{k \to \infty}\EE[S \sim \binom{V_k}{\ctp{\theta n k}}]{\mu_k\left[\Ent[S]{f^k}\right]}\\
&\overset{\text{by~\eqref{eq-rela-2}}}{=}  	\frac{Z_\pi}{\theta^{n}} \E[R \sim \Bin(V,1-\theta)]{\pi_{R}\tp{\mathds{1}_{R}} \cdot \Ent[\pi^{\mathds{1}_{R}}]{f}}.
\end{align*}
This proves the magnetized block factorization of entropy for $\mu$.

\section{Definitions and Preliminaries}

\subsection{Markov chain background}

Let $\Omega$ be a finite state space and $(X_t)_{t \geq 0}$ is a Markov chain on it.
The Markov chain $(X_t)_{t \geq 0}$ can be represented by a transition matrix $P \in \mathds{R}_{\geq 0}^{\Omega \times \Omega}$. We often use the matrix $P$ to refer to $(X_t)_{t \geq 0}$ when the context is clear.
A distribution $\mu$ is called a stationary distribution of $P$ if $\mu = \mu P$.
A Markov chain is said to be
\begin{itemize}
\item \emph{irreducible}, if for any $X, Y \in \Omega$, there is an integer $t \geq 0$ such that $P^t(X, Y) > 0$;
\item \emph{aperiodic}, if for any $X \in \Omega$, $\mathrm{gcd}\{t > 0 \mid P^t(X, X) > 0\} = 1$;
\item \emph{reversible} with respect to $\mu$, if the following \emph{detailed balanced equation} is satisfied
  \begin{align*}
    \forall X, Y \in \Omega, \quad \mu(X)P(X, Y) &= \mu(Y)P(Y, X),
  \end{align*}
  which also implies that $\mu$ is a stationary distribution of $P$.
\end{itemize}
It is well-known that when a Markov chain is both irreducible and aperiodic, then it has a unique stationary distribution~\cite{levin2009markov}.

Let $\mu$ be a distribution with support $\Omega$ and $P$ be a Markov chain over $\Omega$ with the unique stationary distribution $\mu$.
To measure the convergence rate of $P$, we define the \emph{mixing time} of $P$ to be
\begin{align*}
  \forall 0 < \epsilon < 1,\quad T_{\mathrm{mix}}(\epsilon) &\triangleq \max_{X \in \Omega} \min \left\{ t \mid \DTV{P^t(X, \cdot)}{\mu} \leq \epsilon \right\},
\end{align*}
where $\DTV{P^t(X, \cdot)}{\mu}$ is the \emph{total variation distance} between $P^t(X, \cdot)$ and $\mu$, which is defined as
\begin{align*}
  \DTV{P^t(X, \cdot)}{\mu} &\triangleq \frac{1}{2} \sum_{Y \in \Omega} \abs{P^t(X, Y) - \mu(Y)}.
\end{align*}

There is a sharp connection between the mixing time and the \emph{modified log-Sobolev (MLS) constant}.
To introduce it, we define the \emph{Dirichlet form} with respect to $P$ as
\begin{align*}
  \+E_P(f, g) \triangleq \inner{f}{(I - P) g}_\mu \triangleq \sum_{X \in \Omega} \mu(X)\; f(X)\; (I - P)g(X),
\end{align*}
where $f, g$ are two functions over $\Omega$ and $I$ denotes the identity matrix.
If the Markov chain $P$ is reversible with respect to $\mu$, the Dirichlet form can be written as
\begin{align*}
\+E_P(f, g) = \frac{1}{2}\sum_{\sigma,\tau \in \Omega}\mu(\sigma)P(\sigma,\tau)(f(\sigma)-f(\tau))(g(\sigma)-g(\tau)).	
\end{align*}
Moreover, let the \emph{entropy} of a function $f: \Omega \to \mathds{R}_{\geq 0}$ as
\begin{align*}
  \Ent[\mu]{f} &\triangleq \EE[\mu]{f \log f} - \EE[\mu]{f} \log \EE[\mu]{f},
\end{align*}
where we use the convention that $0 \log 0 = 0$.
Note that when $\EE[\mu]{f} = 1$, then $\Ent[\mu]{f}$ is exactly the \emph{relative entropy} (a.k.a. \emph{KL-divergence}) of $\nu(\cdot)=\mu(\cdot)f(\cdot)$ and $\mu(\cdot)$. Formally
\begin{align*}
\Ent[\mu]{f} = \KL{\nu}{\mu} \triangleq \sum_{\sigma \in \Omega}\nu(\sigma)\log \frac{\nu(\sigma)}{\mu(\sigma)}, \quad \text{where } \nu(\sigma) = \mu(\sigma)f(\sigma).	
\end{align*}
For any functions $\hat{f},\hat{g}$ over $\hat{\Omega}$ such that $\hat{\Omega} \supseteq \Omega$, we simply use $\+E_P(\hat{f},\hat{g})$ and $\Ent[\mu]{\hat{f}}$ to denote $\+E_P(\hat{f}_\Omega,\hat{g}_\Omega)$ and $\Ent[\mu]{\hat{f}_\Omega}$, where $\hat{f}_\Omega$ and $\hat{g}_\Omega$ are obtained by restricting $\hat{f}$ and $\hat{g}$ on $\Omega$.
The \emph{modified log-Sobolev constant} introduced in~\cite{bobkov2006modified}  is defined by:
\begin{align*}
  \rho(P) &\triangleq \inf \left\{ \left.\frac{\+E_{P}(f, \log f)}{\Ent[\mu]{f}} \;\right\vert\; f:\Omega(\mu) \to \mathds{R}_{\geq 0},\; \Ent[\mu]{f} \not= 0 \right\}.
\end{align*}
If $P$ is reversible and all the eigenvalues of $P$ are non-negative, then by \cite[Corollary 2.8]{bobkov2006modified} and \cite[Corollary 2.2, (ii)]{diaconis1996log}, it holds that
\begin{align}\label{eq:MLS-mixing}
  t_{\mathrm{mix}}(P, \epsilon) &\leq \frac{1}{\rho(P)} \tp{\log \log \frac{1}{\mu_{\min}} + \log \frac{1}{2\epsilon^2}},
\end{align}
where $\mu_{\min} = \min_{\sigma \in \Omega}\mu(\sigma)$ denotes the minimum probability in $\mu$.
Specifically, \cite{bobkov2006modified} gives the mixing time for the continuous time variant of $P$, and \cite{diaconis1996log} shows that when $P$ is reversible and all the eigenvalues of $P$ are non-negative, then the mixing time of $P$ could be bounded by the mixing time of its continuous time variant.

%Bobkov and Tetali[BT06\todo{cite needed}] showed the following relation between MLS constant and mixing time.
%\begin{align}\label{eq:MLS-mixing}
%  T_{\mathrm{mix}}(P^{\mathrm{GD}}_\mu,\epsilon) \le \frac{1}{\rho^{\mathrm{GD}}(\mu)} \tp{\log \log \frac{1}{\mu_{\mathrm{min}}} + \log \frac{1}{2 \epsilon^2}}.
%\end{align}

%Let $V$ be a ground set of size $[n]$.
%
%Without loss of generality, we assume $V = [n] = \{1,2,\ldots,n\}$.
%
\subsection{Fractional log-concavity and correlation matrix}
Let $V$ be a finite set.
Without loss of generality, we assume that $V = [n] = \{1,2,\ldots,n\}$.
%
%Let $[n] \triangleq \{1,2,\ldots,n\}$ be a finite set.
%
Let $\mu$ be a distribution over $\{\0,\1\}^V$.
Equivalently, we can view $\mu$ as a distribution over the subsets of $[n]$ such that
for any \emph{configuration} $\sigma \in \{\0,\1\}^{[n]}$,
\begin{align}
\label{eq-distribution-equal}
\mu(\sigma) = \mu(S_\sigma), \quad\text{where } S_\sigma = \{i \in [n] \mid \sigma_i = \1\}.
\end{align}
For any $i \in [n]$, we may use $i$ to denote event $i \in S$, and $\overline{i}$ to denote event $i \notin S$.
The \emph{signed influence matrix} for $\mu$ in \Cref{definition-inf-matrix} can be rewritten as follows
%\begin{definition}[signed influence matrix]
%Let $\mu:2^{[n]} \to \mathds{R}_{\geq 0}$ be a distribution. 
%The influence matrix $\Ima_{\mu}: [n]\times [n] \to \mathds{R}$ is defined by
\begin{align*}
\forall i,j \in [n], \quad \Ima_{\mu}(i,j) = \begin{cases}
	\Pr[\mu]{j \mid i} - \Pr[\mu]{j \mid \overline{i}} &\text{if } i \neq j \land \Pr[\mu]{i} > 0 \land \Pr[\mu]{\overline{i}} > 0;\\
 	0 &\text{otherwise.}
 \end{cases}
\end{align*}	
%\end{definition}
%\begin{definition}[magnetizing a joint distribution with local fields]
%Let $\mu:2^{[n]} \to \mathds{R}_{\geq 0}$ be a distribution.
Let $\*\lambda = (\lambda_i)_{i \in [n]} \in \mathds{R}_{>0}$ be local fields.
Let $\mu^{(\*\lambda)}$ denote the distribution obtained from imposing the local fields $\*\lambda$ onto $\mu$, which can also be viewed as a distribution over subsets of $[n]$.
Formally, %for any subset $S \subseteq [n]$,
\begin{align*}
\forall S \subseteq [n],\quad 
\mu^{(\*\lambda)}(S) \propto \mu(S) \prod_{i \in S}\lambda_i.	
\end{align*}
%\end{definition}

%\begin{definition}[spectral independence with arbitrary local fields]
%\label{definition-SI-local}
%Let $C>0$ be a parameter.
%A distribution $\mu:2^{[n]} \to \mathds{R}_{\geq 0}$ is said to be $C$-\emph{spectrally independent with arbitrary local fields} if 
%$\lambda_{\max} (\Ima_{\mu^{(\*\lambda)}}) \leq C$ for any  positive local fields $\*\lambda = (\lambda_i)_{i \in [n]} \in \mathds{R}_{>0}$.
%%$\lambda_{\max}(\Psi^{\mathrm{inf})$	
%\end{definition}

%
The \emph{generating polynomial} associated to $\mu$ is defined by
\begin{align*}
g_\mu (z_1,z_2,\ldots,z_n) \triangleq \sum_{S \subseteq [n]}\mu(S)\prod_{i \in S}z_i.	
\end{align*}
\begin{definition}[fractional log-concavity\text{\cite{AASV21}}]
Let $\alpha \in [0,1]$.
A distribution 	$\mu: 2^{[n]} \to \mathds{R}_{\geq 0}$ is said to be $\alpha$-fractionally log-concave ($\alpha$-FLC) if $\log g_{\mu}(z_1^\alpha,z_2^\alpha,\ldots,z_n^\alpha)$ is concave, viewed as a function on $\mathds{R}_{\geq 0}^n$.
\end{definition}

The notion of fractional log-concavity is closely related to the following \emph{correlation matrix}.
\begin{definition}[correlation matrix]
\label{definition-cor-matrix}
Let $\mu:2^{[n]} \to \mathds{R}_{\geq 0}$ be a distribution. 
The correlation matrix $\Cma_{\mu}: [n]\times [n] \to \mathds{R}$ is defined by
\begin{align*}
\forall i,j \in [n], \quad \Cma_{\mu}(i,j) = \begin{cases}
 	1 - \Pr[\mu]{i} &\text{if } i = j;\\
 	\Pr[\mu]{j \mid i} - \Pr[\mu]{j} &\text{if } i \neq j \land \Pr[\mu]{i} > 0;\\
 	0 & \text{if } i \neq j \land \Pr[\mu]{i} = 0.
 \end{cases}
\end{align*}

\end{definition}

\begin{proposition}[\text{\cite{AASV21}}]
\label{proposition-cor-FLC}
Let $\mu:2^{[n]} \to \mathds{R}_{\geq 0}$ be a distribution and $\alpha \in (0,1]$.
$\mu$ is $\alpha$-fractionally log-concave if and only if $\lambda_{\max}(\Cma_{\mu^{(\lambda)}}) \leq \frac{1}{\alpha}$ for any 	positive local fields $\*\lambda = (\lambda_i)_{i \in [n]} \in \mathds{R}_{>0}$.
\end{proposition}

For the proof of \Cref{proposition-cor-FLC}, readers can refer to the proof of Lemma~69 and Remark~70 in the full version of~\cite{AASV21}.

\subsection{Homogeneous distributions and random walks}
Let $\mu:\binom{[n]}{k} \to \mathds{R}_{\geq 0}$ be a distribution over the size-$k$ subsets of $[n]$. We call such distributions the \emph{homogeneous distributions}.

Let $\Omega \subseteq \binom{[n]}{k}$ denote the support of $\mu$.
Let $X$ be the downward closure of $\Omega$.  
Formally, $X$ is the smallest family such that $\Omega \subseteq X$ and if $\alpha \in X$ then $\beta \in X$ for all $\beta \subseteq \alpha$.
In other words, $X$ is the \emph{simplicial complexes} generated by $\mu$.
For any \emph{face} $\alpha \in X$, let $\abs{\alpha}$ denote the \emph{dimension} of $\alpha$.
For any integer $0 \leq j \leq  k$, let $X(j)$ denote all the faces in $X$ with dimension $j$.
\begin{definition}[down/up walk]
\label{definition-DUW}
Let $X$ be the simplicial complexes generated by a homogeneous distribution $\mu:\binom{[n]}{k} \to \mathds{R}_{\geq 0}$.
%Let $\mu:\binom{[n]}{k} \to \mathds{R}_{\geq 0}$ be a homogeneous distribution.
Let $0 \leq j < k$ be an integer.
%The down walk and up walk are defined as follows.
\begin{itemize}
\item The down walk $D_{k \to j}: X(k) \times X(j) \to \mathds{R}_{\geq 0}$ is defined by
\begin{align*}
\forall \alpha \in X(k),\beta \in X(j),\quad D_{k \to j}(\alpha,\beta) = \begin{cases}
 	\frac{1}{\binom{k}{j}} &\text{if } \beta \subseteq \alpha;\\
 	0 &\text{otherwise.}
 \end{cases}
\end{align*}
\item The up walk $U_{j \to k}: X(j) \times X(k) \to \mathds{R}_{\geq 0}$ is defined by
\begin{align*}
\forall \alpha \in X(j),\beta \in X(k),\quad U_{j \to k}(\alpha,\beta) = \begin{cases}
 	\frac{\mu(\beta)}{\sum_{\gamma \in X(k): \alpha \subseteq \gamma}\mu(\gamma)} &\text{if } \alpha \subseteq \beta;\\
 	0 &\text{otherwise.}
 \end{cases}
\end{align*}
\end{itemize}
%For any integers $0 \leq j < k \leq n$, the down walk $D_{k \to j}$ from  $\binom{[n]}{k}$ to $\binom{[n]}{j}$ is defined as follows. Given any $S \subseteq \binom{[n]}{k}$, the down walk $D_{k \to \ell}$ transforms $S$ to a subset $T \in \binom{S}{j}$ uniformly at random.
\end{definition}

The following relative entropy decay result is proved in~\cite{anari2021entropic}.

\begin{theorem}[\text{\cite[Theorem~4]{anari2021entropic}}]
\label{theorem-AJK+}
Suppose $\mu:\binom{[n]}{k} \to \mathds{R}_{\geq 0}$ is $\alpha$-fractionally log-concave for some $\alpha \in (0,1]$. 
Let $\Omega$ denote the support of $\mu$.
For any integer $0 \leq j \leq k - \lceil 1/\alpha \rceil$, and any distribution $\nu$ over $\Omega$,
\begin{align*}
\KL{\nu D_{k \rightarrow j}}{ \mu D_{k \rightarrow j} }	\leq \tp{1-\kappa\tp{j,k,\frac{1}{\alpha}}}\KL{\nu}{\mu},
\end{align*}
where
\begin{align}
\label{eq-def-kappa}
\kappa\tp{j,k,c} \triangleq \frac{(k+1-j-c)^{c- \ctp{c} } \prod_{i=0}^{\ctp{c}-1}(k-j - i) }{(k+1)^{c}}.
\end{align}
\end{theorem}

%For any distribution $\nu$ with support $\Omega_\nu$, any non-negative function $g: \Omega_{\nu} \to \mathds{R}_{\geq 0}$, define the following entropy-like quantity 
%$\Ent[\nu]{g} = \E[\nu]{g \log g} - \E[\nu]{g} \log \E[\nu]{g},$
%where we use convention $0\log 0 = 0$.
Let $\mu_{(k)} = \mu$. For any $0 \leq j < k$, let $\mu_{(j)} = \mu_{(k)} D_{k\to j}$.
Let $f^{(k)}: X(k) \to \mathds{R}_{\geq 0}$. 
For any $0 \leq j < k$, define the function $f^{(j)}: X(j) \to \mathds{R}_{\geq 0}$ by $f^{(j)} = U_{j \to k}f^{(k)}$. 
We have the following lemma.

\begin{lemma}
\label{lemma-local-etp-decay}
Let  $\mu=\mu_{(k)}:\binom{[n]}{k} \to \mathds{R}_{\geq 0}$, $0 \leq j < k$ and $\kappa \in (0,1)$. Let $\Omega$ denote the support of $\mu$.
Suppose for any distribution $\nu$ over $\Omega$, $\KL{\nu D_{k \rightarrow j}}{ \mu D_{k \rightarrow j} }	\leq (1-\kappa)\KL{\nu}{\mu}$.
For any function $f^{(k)}: X(k) \to \mathds{R}_{\geq 0}$,
\begin{align*}
\Ent[\mu_{(j)}]{f^{(j)}} \leq (1-\kappa)	\Ent[\mu_{(k)}]{f^{(k)}}.
\end{align*}
\end{lemma}
\begin{proof}
First note that if $f^{(k)} = 0$, then the lemma holds trivially.
We next prove that we only need to consider the function $f^{(k)}$ with $\E[\mu_{(k)}]{f^{(k)}} = 1$. 
For any $g: X(k) \to \mathds{R}_{\geq 0}$ and any $c > 0$, it holds that
$\Ent[\mu_{(k)}]{cg} = c\Ent[\mu_{(k)}]{g}$ and $	\Ent[\mu_{(j)}]{(cg)^{(j)}} = c \Ent[\mu_{(j)}]{g^{(j)}}$
where $g^{(j)} = U_{j \to k}g$ and $(cg)^{(j)} = U_{j \to k}(cg)$ for $j < k$. 
Suppose $f^{(k)} \neq 0$ and $t = \E[\mu_{(k)}]{f^{(k)}} > 0$. The lemma holds for $f^{(k)}$ if and only if the lemma holds for $f' = f^{(k)}/t$ and $\E[\mu_{(k)}]{f'} = 1$.

Assume  $\E[\mu_{(k)}]{f^{(k)}} =1$. Define a distribution $\nu$ over $X(k)$ by $\nu(\sigma) = \mu(\sigma)f^{(k)}(\sigma)$ for all $\sigma \in X(k)$, then
\begin{align*}
\KL{\nu}{\mu}  = \sum_{\sigma \in X(k)}\mu(\sigma)\frac{\nu(\sigma)}{\mu(\sigma)} \log \frac{\nu(\sigma)}{\mu(\sigma)} = \EE[\mu]{f^{(k)} \log f^{(k)}} \overset{(\ast)}{=}  \Ent[\mu_{(k)}]{f^{(k)}},
\end{align*}
where $(\ast)$ holds because $\E[\mu]{f^{(k)}} = 1$ and $\mu=\mu_{(k)}$. 
Let $\nu_{(j)} = \nu D_{k\to j}$ and $\mu_{(j)} = \mu D_{k \to j}$.
We have
\begin{align*}
\KL{\nu D_{k \rightarrow j}}{ \mu D_{k \rightarrow j} }	= \sum_{\sigma \in X(j)}	\mu_{(j)}(\sigma)\frac{\nu_{(j)}(\sigma)}{\mu_{(j)}(\sigma)}\log\frac{\nu_{(j)}(\sigma)}{\mu_{(j)}(\sigma)} = \Ent[\mu_{(j)}]{f^{(j)}},
\end{align*}
where the last equation holds because $f^{(j)}(\sigma) = U_{j \to k}f(\sigma) =  \sum_{\alpha \in X(k): \sigma \subseteq \alpha}U_{j \to k}(\sigma,\alpha)\frac{\nu(\alpha)}{\mu(\alpha)} =\frac{\nu_{(j)}(\sigma)}{\mu_{(j)}(\sigma)}$ and $\E[\mu_{(j)}]{f^{(j)}} = \sum_{\sigma \in X(j)}\mu_{(j)}(\sigma)\frac{\nu_{(j)}(\sigma)}{\mu_{(j)}(\sigma)} = 1$.
%\begin{align*}
%f^{(j)}(\sigma)& = U_{j \to k}f(\sigma) = \sum_{\alpha \in X(k): \sigma \subseteq \alpha}U_{j \to k}(\sigma,\alpha)\frac{\nu(\alpha)}{\mu(\alpha)} = \frac{\sum_{\alpha \in X(k): \sigma \subseteq \alpha}\nu(\alpha)}{\sum_{\beta \in X(k): \sigma \subseteq \beta}\mu(\beta)} = \frac{\nu_j(\sigma)}{\mu_j(\sigma)},\\
%%\frac{\nu_j(\sigma)}{\mu_j(\sigma)} = 	
%\E[\mu_j]{f^{(j)}} &= \sum_{\sigma \in X(j)}\mu_j(\sigma)\frac{\nu_j(\sigma)}{\mu_j(\sigma)} = 1.
%\end{align*}
%We have
%\begin{align*}
%\KL{\nu D_{k \rightarrow j}}{ \mu D_{k \rightarrow j} } = \Ent[\mu_j]{f^{(j)}}.
%\end{align*}
This proves the lemma.
\end{proof}

\subsection{Multivariate hypergeometric distribution}
Let $V$ be a set of $n$ buckets, each of them has $k$ balls.
Suppose we pick $\ell$ balls from all $kn$ balls uniformly at random, without replacement.
For each bucket $v \in V$, let $a_v \in \mathbb{Z}_{\geq 0}$ denote the number of balls picked from the bucket $v$,
then $\*a = (a_v)_{v \in V}$ follows multivariate hypergeometric distribution.

Formally, given a set $V$ of size $n$, an integer $k \geq 1$ and an integer $0\leq \ell \leq kn$, the multivariate hypergeometric distribution $\HyperGeo$ is defined as follows.
The support of  $\HyperGeo$ is defined by
\begin{align}
\label{eq-support-hypergeo}
  \Omega(\HyperGeo) \triangleq \left\{\boldsymbol{a} = (a_v)_{v \in V} \mid  \sum_{v \in V} a_v = \ell \mbox{ and } \forall v \in V , a_v \in \mathbb{Z}_{\geq 0} \right\},
\end{align}
%where $a_i$ stands for the number of balls we have picked from the $i$-th bucket.
For any $\boldsymbol{a} \in \Omega(\HyperGeo)$, it holds that
\begin{align}
\label{eq-hypergeo}
  \HyperGeo (\boldsymbol{a}) = \frac{\prod_{v \in V} \binom{k}{a_v}}{\binom{kn}{\ell}}.
\end{align}

%The random variables $(X_{(v,i)})_{v \in V, i \in [k]}$ are negatively associated \cite{joag-dev1983}, and hence the Chernoff-Hoeffding inequality is applicable.
\begin{lemma}[\cite{joag-dev1983} and \cite{dubhashi1998balls}]\label{lemma:hypergometric-concentration}
Let $\*a \sim  \HyperGeo$. 
For any $v \in V$ and $\epsilon \in (0,1)$, it holds that
\begin{align*}
  \Pr{\left\vert \frac{a_v}{k} - \frac{\ell}{kn} \right\vert \geq \epsilon} &\leq 2\exp\tp{-2\epsilon^2 k}.
\end{align*}
\end{lemma}
\begin{proof}
For each bucket $v \in V$, let $(v,1),(v,2),\ldots,(v,k)$ denote all balls in $v$.
For each ball $(v,i)$, let $X_{(v,i)}$ indicate whether the ball $(v,i)$ is picked.
It holds that $a_v = \sum_{i \in [k]}X_{(v,i)}$.
Since $(X_{(v,i)})_{i \in [k]}$ are negative associated~\cite[Lemma 2.11]{joag-dev1983}, the Chernoff-Hoeffding inequality \cite{dubhashi1998balls} can be applied to $a_v$.
\end{proof}

%Let $n \geq 1, k \geq 1$ be integers.
%Given an integer $1 \leq \ell \leq kn$, let $\+S\in \binom{[kn]}{\ell}$ be a random subset sampled from $\binom{[kn]}{\ell}$ uniformly at random.
%
%For each $1 \le i \le kn$, define the indicator random variable $X_i = \mathbf{1}[i \in \+S]$. Let $X = \sum_{i=1}^k X_i$.
%For any $\epsilon \in (0, 1)$ it holds that
%\begin{align*}
%  \Pr{\left\vert \frac{X}{k} - \frac{\ell}{kn} \right\vert \geq \epsilon} &\leq 2\exp\tp{-2\epsilon^2 k}.
%\end{align*}
%\end{lemma}

%For any $i \in [kn] \triangleq \{1,2,3,\ldots,kn\}$,  we use $X_i$ to indicate whether the $i$-th ball is picked. The random variables $X_1, X_2, \cdots, X_{kn}$ are negatively associated \cite{joag-dev1983}, and hence the Chernoff-Hoeffding inequality is applicable.

%\begin{lemma}[\cite{joag-dev1983} and \cite{dubhashi1998balls}]\label{lemma:hypergometric-concentration}
%Let $n \geq 1, k \geq 1$ be integers.
%Given an integer $1 \leq \ell \leq kn$, let $\+S\in \binom{[kn]}{\ell}$ be a random subset sampled from $\binom{[kn]}{\ell}$ uniformly at random.
%
%For each $1 \le i \le kn$, define the indicator random variable $X_i = \mathbf{1}[i \in \+S]$. Let $X = \sum_{i=1}^k X_i$.
%For any $\epsilon \in (0, 1)$ it holds that
%\begin{align*}
%  \Pr{\left\vert \frac{X}{k} - \frac{\ell}{kn} \right\vert \geq \epsilon} &\leq 2\exp\tp{-2\epsilon^2 k}.
%\end{align*}
%\end{lemma}

\section{Fractional Log-Concavity and Uniform Block Factorization}
\label{section-proof-AJK+}
In this section, we prove \Cref{corollary-muk-UBF}.
The lemma can be proved by the techniques in \cite{AASV21,anari2021entropic}. We conclude the proof here for completeness. 
%Let $V = [n]$ be a finite set.
%Let $\mu$ be a distribution over $\{\0,\1\}^{V}$.
%By~\eqref{eq-distribution-equal}, we can view $\mu$ as a distribution over $2^V$.
We prove a general result based on spectral independence.

\begin{definition}[spectral independence with all fields]
\label{definition-SI-local}
Let $\eta >0$.
A distribution $\mu$ over $\{\0,\1\}^V$ is said to be $\eta$-\emph{spectrally independent with all fields} if 
$\lambda_{\max} (\Ima_{\mu^{(\*\lambda)}}) \leq \eta$ for any $\*\lambda = (\lambda_v)_{v \in V} \in \mathds{R}_{>0}^V$.
%$\lambda_{\max}(\Psi^{\mathrm{inf})$	
\end{definition}
The notion of spectral independence was first introduced in~\cite{anari2020spectral}, then further developed in~\cite{chen2021rapidcolor,feng2021rapid}.
Remark that if a distribution is $\eta$-spectrally independent in $\infty$-norm  with all fields, then it must be $\eta$-spectrally independent with all fields.
This is because $\lambda_{\max} (\Ima_{\mu^{(\*\lambda)}}) \leq \norm{\Ima_{\mu^{(\*\lambda)}}}_{\infty}$.

\begin{proposition}[\text{\cite{AASV21,anari2021entropic}}]
\label{proposition-general-SI-UBF}
Let $V= [n] =\{1,2,\ldots,n\}$ and $\eta > 0$.
For any distribution $\mu$ over $\{\0,\1\}^V$ that is $\eta$-spectrally independent with all fields, any integer $\ctp{\eta + 1} < \ell \leq n$, $\mu$ satisfies $\ell$-uniform block factorization of entropy with $C = 1/\kappa(n-\ell,n,\eta + 1)$, where $\kappa(\cdot)$ is defined in~\eqref{eq-def-kappa}.
\end{proposition}
Now, we are ready to prove \Cref{corollary-muk-UBF}.

\begin{proof}[Proof of \Cref{corollary-muk-UBF}]
Consider the distribution $\mu_k$. Since $\mu_k$ is $\eta$-spectrally independent in $\infty$-norm with all fields, $\mu_k$ is $\hat{\eta}$-spectrally independent with all fields, where $\hat{\eta} = \ctp{\eta}$.
Let $\ell = \ctp{ \theta n k}$.
Since  $k > \frac{\eta + 2}{\theta n}$, it holds that $\ctp{\hat{\eta} + 1} <\ell \leq kn$.
By \Cref{proposition-general-SI-UBF}, $\mu_k$ satisfies $\ell$-uniform block factorization of entropy with $C = 1/\kappa(nk-\ell,nk,\hat{\eta} + 1)$. Note that $\hat{\eta} + 1$ is an integer. As observed in \cite{anari2021entropic}, $\kappa(nk-\ell,nk,\hat{\eta} + 1)$ defined in~\eqref{eq-def-kappa} equals to $\binom{\ell}{\hat{\eta}+1}/\binom{nk}{\hat{\eta} + 1}$. This implies that
\begin{align*}
C = \frac{1}{\kappa(nk-\ell,nk,\hat{\eta} + 1)} = \binom{nk}{\hat{\eta} + 1} \Big/\binom{\ell}{\hat{\eta}+1} \leq \tp{\frac{\mathrm{e}nk}{\hat{\eta}+1}}^{\hat{\eta}+1} \Big/ \tp{\frac{\ell}{\hat{\eta}+1}}^{\hat{\eta}+1} \leq \tp{\frac{\mathrm{e}nk}{\ell}}^{\eta + 2} \leq \tp{\frac{\mathrm{e}}{\theta}}^{\eta + 2}. &\qedhere
\end{align*}

\end{proof}

We now prove \Cref{proposition-general-SI-UBF}. We need to introduce some definitions.
By~\eqref{eq-distribution-equal}, we can view $\mu$ as a distribution over $2^{[n]}$ such that for any $\sigma \in \{\0,\1\}^{[n]}$, $\mu(S_\sigma)=\mu(\sigma)$, where $S_\sigma = \{i \in [n] \mid \sigma_i = \1\}$.
%We often view the distribution $\mu$ over $\{\0,\1\}^V$ as a distribution over subsets of $[n]$ such that for any $\sigma \in \{\0,\1\}^V$, $\mu(S_\sigma)=\mu(\sigma)$, where $S_\sigma = \{i \in [n] \mid \sigma_i = \1\}$.
Define the \emph{homogenization} of a distribution $\mu$ over $\{\0,\1\}^{[n]}$ is a distribution $\pi = \mu^{\mathrm{hom}}$ over subsets of $[n] \cup \overline{[n]} = \{1,2,\ldots,n\} \cup \{\overline{1},\overline{2},\ldots,\overline{n}\}$.
For any $S \subseteq [n]$, define $S^c = \{\overline{i} \mid i \in [n] \setminus S \}$, and let $\pi(S\cup S^c) = \mu(S)$.
It is straightforward to verify $\pi$ is a homogeneous distribution over $\binom{[n] \cup \overline{[n]}} {n}$.
Let $\Cma_{\pi}$ denote the correlation matrix (\Cref{definition-cor-matrix}) of $\pi$ and $\Ima_{\mu}$ denote the signed influence matrix (\Cref{definition-inf-matrix}) of $\mu$.
The following result is proved in~\cite{AASV21}~(see the proof of Lemma~71 in the full version of~\cite{AASV21}).
\begin{lemma}[\text{\cite{AASV21}}]
\label{lemma-sp-AASV}
The spectrum of $\Cma_{\pi}$ is the union of $\{\lambda_i + 1\}_{1 \leq i \leq n}$ and $n$ copies of 0, where $\lambda_1\geq \lambda_2\geq \ldots \geq \lambda_n$ are all eigenvalues of $\Ima_{\mu}$.
\end{lemma}

To give the next lemma, we view $\mu$ as a distribution over $\{\0,\1\}^{[n]}$.
We use $\Omega(\mu)$ and $\Omega(\pi)$ to denote the support of $\mu$ and $\pi$ respectively.
Fix a function $f: \Omega(\mu) \to \mathds{R}_{\geq 0}$.
Recall  $S_{\sigma} = \{i \in [n] \mid \sigma_i = \1\}$ for any $\sigma \in \{\0,\1\}^{[n]}$.
%Let  $\Omega_{\mu}$ denote the support of $\mu$.
We can construct $f^{(n)}: \Omega({\pi}) \to \mathds{R}_{\geq 0}$ by $f^{(n)}(S_\sigma \cup S^c_\sigma) = f(\sigma)$ for all $\sigma \in \Omega_{\mu}$, where $S^c_\sigma = \{\overline{i} \mid i \in [n] \setminus S_\sigma \}$.
Let $X$ denote the simplicial complexes generated by $\pi$.
Let $U_{\cdot}$ and $D_{\cdot}$ denote the up walk and down walk on $X$ (\Cref{definition-DUW}).
Let $\pi_{(n)} = \pi$ and $\pi_{(j)} = \pi_{(n)} D_{n \to j}$ for all $0 \leq j < n$.
%
%Let $\Omega_{\pi}$ denote the support of $\pi$.
Let $f^{(j)} = U_{j \to n}f^{(n)}$ for all $0 \leq j < n$.
The following lemma is proved in~\cite{chen2020optimal} (see the proof of Lemma~2.6 in the full version of~\cite{chen2020optimal}).
\begin{lemma}[\text{\cite{chen2020optimal}}]
\label{lemma-UBF=ED}
For any $0\leq j \leq n$,
it holds that 
\begin{align}
\frac{1}{\binom{n}{j}}\sum_{S \in \binom{[n]}{j}}\mu[\Ent[S]{f}]&= \Ent[\pi_{(n)}]{f^{(n)}} - \Ent[\pi_{(n-j)}]{f^{(n-j)}}.\label{eq-CLV-2}	
\end{align}
\end{lemma}

Now, we are ready to prove \Cref{proposition-general-SI-UBF}.

\begin{proof}[Proof of \Cref{proposition-general-SI-UBF}]
%We first introduce some definitions.
We view $\mu$ as a distribution over $2^{[n]}$.
%
%Let $\mu(S_\sigma) = \mu(\sigma)$.
%
Let $\pi = \mu^{\mathrm{hom}}$ be the homogenization of $\mu$. 

We first show that $\pi$ is $\frac{1}{\eta+1}$-fractionally log-concave. 
By~\Cref{proposition-cor-FLC}, we only need to show that $\lambda_{\max}(\Cma_{\pi^{(\lambda)}}) \leq \eta+1$ for any 	positive local fields $\*\lambda = (\lambda_i)_{i \in [n] \cup \overline{[n]} } \in \mathds{R}_{>0}^{[n] \cup \overline{[n]}}$.
Fix a $\*\lambda  \in \mathds{R}_{>0}^{[n] \cup \overline{[n]}}$. 
Define $\*\phi = (\phi_i)_{i \in [n]}$ such that for any $i \in [n]$, $\phi_i = \frac{\lambda_i}{\lambda_{\overline{i}}}$.	
For any subset $S \subseteq [n]$, recall $S^c = \{\overline{i} \mid i \in [n] \setminus S \}$, we have 
\begin{align*}
\mu^{(\*\phi)}(S) &\propto \mu(S)\prod_{i \in S} 	\frac{\lambda_i}{\lambda_{\overline{i}}} \propto \tp{\prod_{i \in [n]}\lambda_{\overline{i}} }\cdot \mu(S)\prod_{i \in S} 	\frac{\lambda_i}{\lambda_{\overline{i}}} = \mu(S) \prod_{i \in S}\lambda_i \prod_{i \in [n] \setminus S}\lambda_{\overline{i}},\\
\pi^{(\*\lambda)}(S \cup S^c) &\propto \mu(S) \prod_{i \in S}\lambda_i \prod_{i \in [n] \setminus S}\lambda_{\overline{i}} \propto \mu^{(\*\phi)}(S).
\end{align*}
This implies that $\pi^{(\*\lambda)}$ is the homogenization of $\mu^{(\*\phi)}$.
Since $\mu$ is $\eta$-spectrally independent with all fields, it holds that $\lambda_{\max}(\Ima_{\mu^{(\*\phi)}})\leq \eta$.
By \Cref{lemma-sp-AASV}, we have $\lambda_{\max}(\Cma_{\pi^{(\*\lambda)}})\leq \eta+1$.
By~\Cref{proposition-cor-FLC}, $\pi$ is $\frac{1}{\eta+1}$-fractionally log-concave.

%By \Cref{theorem-AJK+}, the distribution $\pi$ satisfies the relative entropy decay property.
%
Let $X$ denote the simplicial complexes generated by $\pi$.
Let $U_{\cdot}$ and $D_{\cdot}$ denote the up walk and down walk on $X$ (\Cref{definition-DUW}).
Let $\pi_{(n)} = \pi$ and $\pi_{(j)} = \pi_{(n)} D_{n \to j}$ for all $0 \leq j < n$.
Let $\Omega({\pi})$ denote the support of $\pi$.
Let $f^{(n)}: \Omega(\pi) \to \mathds{R}_{\geq 0}$ and $f^{(j)} = U_{j \to n}f^{(n)}$ for all $0 \leq j < n$.
Combining \Cref{theorem-AJK+} and \Cref{lemma-local-etp-decay}, we have for any $0\leq j < n - \ctp{\eta+1}$, it holds that
\begin{align}
\label{eq-proof-e-decay}
\Ent[\pi_{(j)}]{f^{(j)}}  \leq \tp{1-\kappa\tp{j,n,\eta+1}}\Ent[\pi_{(n)}]{f^{(n)}},
\end{align}
where $\kappa(\cdot)$ is defined in~\eqref{eq-def-kappa}.

Now, we view $\mu$ as a distribution over $\{\0,\1\}^V$.
We use $\Omega({\mu}) \subseteq \{\0,\1\}^V$ to denote the support of $\mu$.
Fix a function $f: \Omega(\mu) \to \mathds{R}_{\geq 0}$.
Note that for every $\sigma \in \Omega({\mu})$, the set $S_{\sigma} = \{i \in [n] \mid \sigma_i = \1\}$.
%Let  $\Omega_{\mu}$ denote the support of $\mu$.
We can construct $f^{(n)}: \Omega(\pi) \to \mathds{R}_{\geq 0}$ by $f^{(n)}(S_\sigma \cup S^c_\sigma) = f(\sigma)$ for all $\sigma \in \Omega(\mu)$. 
%Let $\pi_n = \pi$ and $\pi_j = \pi_n D_{n \to j}$ for all $j < n$.
By \Cref{lemma-UBF=ED}, for all $j \leq n$,
\begin{align}
\label{eq-proof-UBF-etp}
\frac{1}{\binom{n}{j}}\sum_{S \in \binom{[n]}{j}}\mu[\Ent[S]{f}]&= \Ent[\pi_{(n)}]{f^{(n)}} - \Ent[\pi_{(n-j)}]{f^{(n-j)}}.	
\end{align}
For any $\ctp{\eta+1}< \ell \leq n$, we have 
\begin{align*}
\Ent[\mu]{f} &\overset{(\ast)}{=} \Ent[\pi_{(n)}]{f^{(n)}} 	= \Ent[\pi_{(n)}]{f^{(n)}} -  \Ent[\pi_{(n-\ell)}]{f^{(n-\ell)}} + \Ent[\pi_{(n-\ell)}]{f^{(n-\ell)}}\\
\text{(by~\eqref{eq-proof-e-decay} and~\eqref{eq-proof-UBF-etp})}\quad &\leq \frac{1}{\binom{n}{\ell}}\sum_{S \in \binom{[n]}{\ell}}\mu[\Ent[S]{f}] + \tp{1- \kappa\tp{n-\ell,n,\eta+1}} \Ent[\pi_{(n)}]{f^{(n)}}\\
&\overset{(\star)}{=} \frac{1}{\binom{n}{\ell}}\sum_{S \in \binom{[n]}{\ell}}\mu[\Ent[S]{f}] + \tp{1- \kappa\tp{n-\ell,n,\eta+1}} \Ent[\mu]{f}
\end{align*}
where $(\ast)$ and $(\star)$ hold due to the definitions of $\pi_{(n)}$ and $f^{(n)}$. This proves the proposition. %follows from above inequality.
\end{proof}

\section{Spectral Independence of Transformed Distributions}
\label{section-SI-muk}
In this section, we prove \Cref{lem:muk-inf}.
Let $\eta > 0$ be a real number.
Recall that $\mu$ is a distribution over $\{\0, \1\}^V$ that is $\eta$-spectrally independent with all fields.
For $k \geq 1$ be an integer, $\mu_k$ is the $k$-transformed distribution of $\mu$, and we want to show that $\mu_k$ is $(\eta+1)$-spectrally independent with all fields.
If $k = 1$, the lemma is trivial. We assume $k \geq 2$ in the proof.

For each $(v,i) \in V \times [k]$, we use $v_i$ to denote the pair $(v,i)$.
Let $\*\phi \in \mathds{R}_{> 0}^{V\times [k]}$ be a vector of local fields and $\pi_k \triangleq \mu_k^{\tp{\*\phi}}$ be the magnetized distribution generated by $\mu_k$ and the local fields $\*\phi$.
Then, let $\*\varphi \in \mathds{R}_{>0}^V$ be another vector of local fields defined as
\begin{align} \label{eq:varphi}
  \forall w\in V, \quad \varphi_w \triangleq \frac{1}{k}\sum_{h \in [k]} \phi_{w_h},
\end{align}
and let $\pi \triangleq \mu^{\tp{\*\varphi}}$ be obtained by imposing the local fields $\*\varphi$ onto $\mu$.

%We use the following lemma to prove \Cref{lem:muk-inf}.

\begin{lemma} \label{lem:muk-inf-aux}
  Let $u_i, v_j \in V \times [k]$, it holds that
  \begin{align}
    \text{if $u \not= v$}, \quad \label{eq:k-trans-inf-target-1}
    & \abs{\Ima_{\pi_k}(u_i, v_j)} \leq \frac{\phi_{v_j}}{\sum_{h \in [k]} \phi_{v_h}} \abs{\Ima_\pi(u, v)} ; \\
    \text{if $u = v$ and $i \not= j$}, \quad \label{eq:k-trans-inf-target-2}
    & \abs{\Ima_{\pi_k}(u_i, u_j)} \leq \frac{\phi_{u_j}}{\sum_{h \in [k] \setminus \{i\}} \phi_{u_h}}.
  \end{align}
\end{lemma}

Now, we are ready to prove \Cref{lem:muk-inf}.
\begin{proof}[Proof of \Cref{lem:muk-inf}]
  Let $u_i$ be an element of $V \times [k]$, by \eqref{eq:k-trans-inf-target-1}, \eqref{eq:k-trans-inf-target-2} and the fact that $\mu$ is $\eta$-spectrally independent with all fields, we have
  \begin{align*}
    \sum_{v_j \in V \times [k]} \abs{\Ima_{\pi_k}(u_i, v_j)} &\leq \sum_{v \in V} \abs{\Ima_{\pi}(u, v)} + 1 \leq \eta + 1. \qedhere
  \end{align*}
\end{proof}

The rest of this section is dedicated to the proof of \Cref{lem:muk-inf-aux}.
Recall that by the definition of $\pi_k$, %it holds that
\begin{align} \label{eq:pi-k}
  \pi_k(\sigma) &= \frac{\mu_k(\sigma)}{Z_{\*\phi}}\prod_{\substack{v_i \in V \times [k] \\ \sigma_{v_i} = \1}} \phi_{v_i},
\end{align}
where $Z_{\*\phi} = \sum_{\sigma \in \Omega(\mu_k)} \mu_k(\sigma) \prod_{v_i \in V \times [k]: \sigma_{v_i} = \1} \phi_{v_i}$ is the normalizing factor.

We prove \Cref{lem:muk-inf-aux} by verifying \eqref{eq:k-trans-inf-target-1} and \eqref{eq:k-trans-inf-target-2}.

\subsection{Verification of \eqref{eq:k-trans-inf-target-1}}
Without loss of generality, we may assume that $\Omega(\pi_{k, u_i}) = \{\0, \1\}$, since otherwise \eqref{eq:k-trans-inf-target-1} holds trivially.
In this case, it is equivalent for us to bound
\begin{align*}
  \abs{\Ima_{\pi_k}(u_i, v_j)} &= \abs{\pi^{u_i\gets \1}_{k,v_j}(\1) - \pi^{u_i\gets \0}_{k,v_j}(\1)}.
\end{align*}
We will calculate $\pi^{u_i \gets \1}_{k, v_j}(\1)$ and $\pi^{u_i \gets \0}_{k, v_j}(\1)$, respectively.
For $\pi^{u_i \gets \1}_{k, v_j}(\1)$, by the law of conditional distirbution, it holds that
\begin{align*}
  \pi^{u_i \gets \1}_{k, v_j}(\1) &= \frac{\Pr[Y\sim \pi_k]{Y_{u_i} = \1 \land Y_{v_j} = \1}}{\Pr[Y \sim \pi_k]{Y_{u_i} = \1}}.
\end{align*}
By the definition of the distribution $\pi_k$, it holds that
\begin{align}
  &\,\Pr[Y \sim \pi_k]{Y_{u_i} = \1 \land Y_{v_j} = \1}
  \nonumber = \frac{1}{Z_{\*\phi}}\sum_{\substack{Y \in \Omega(\mu_k):\\Y_{u_i} = \1 \land Y_{v_j} = \1}} \mu_k(Y) \prod_{\substack{w_h \in V \times [k]:\\ Y_{w_h} = \1}} \phi_{w_h}\\ %\cdot \mathds{1}[Y_{u_i} = \1] \cdot \mathds{1}[Y_{v_j} = \1] \\
  \label{eq:k-trans-inf-1} =\,& \frac{\phi_{u_i}\phi_{v_j}}{Z_{\*\phi}k^2} \sum_{\substack{X \in \Omega(\mu):\\ X_u = \1 \land X_v = \1}} \mu(X)  \prod_{\substack{w \in V \setminus \{u,v\} \\ X_w = \1}} \tp{\sum_{h\in[k]} \frac{\phi_{w_h}}{k}} = \frac{\phi_{u_i}\phi_{v_j}}{Z_{\*\phi}k^2} \sum_{\substack{X \in \Omega(\mu):\\ X_u = \1 \land X_v = \1}} \mu(X)  \prod_{\substack{w \in V \setminus \{u,v\} \\ X_w = \1}} \varphi_w,
\end{align}
where the last equation holds due to the definition of $\*\varphi$ in~\eqref{eq:varphi}.
Similarly, it holds that
\begin{align} \label{eq:k-trans-inf-2}
  \Pr[Y \sim \pi_k]{Y_{u_i} = \1} &= \frac{\phi_{u_i}}{Z_{\*\phi}k} \sum_{\substack{X \in \Omega(\mu):\\ X_u = 1}} \mu(X)\prod_{\substack{w\in V\setminus \{u\} \\ X_w = \1}} \varphi_w.
\end{align}
Combining \eqref{eq:k-trans-inf-1} and \eqref{eq:k-trans-inf-2}, we have
%\begin{align*}
%  \pi^{u_i \gets \1}_{k, v_j}(\1)
%  &= \frac{\sum_{X \in \Omega(\mu)} \mu(X) \prod_{w \in V\setminus \{u, v\}: X_w = \1} \tp{\sum_{h\in[k]} \frac{\phi_{w_h}}{k}} \cdot \mathds{1}[X_u = \1] \frac{\phi_{u_i}}{k} \cdot \mathds{1}[X_v = \1] \frac{\phi_{v_j}}{k}}{\sum_{X \in \Omega(\mu)} \mu(X) \prod_{w\in V\setminus \{v\}: X_w = \1} \tp{\sum_{h\in [k]} \frac{\phi_{w_h}}{k}} \cdot \mathds{1}[X_u = \1] \frac{\phi_{u_i}}{k}} \\
%  &= \frac{\phi_{v_j}}{\sum_{h \in [k]} \phi_{v_h}} \cdot \frac{\sum_{X \in \Omega(\mu)} \mu(X) \prod_{w \in V: X_w = \1} \tp{\sum_{h\in[k]} \frac{\phi_{w_h}}{k}} \cdot \mathds{1}[X_u = \1] \cdot \mathds{1}[X_v = \1] }{\sum_{X \in \Omega(\mu)} \mu(X) \prod_{w\in V: X_w = \1} \tp{\sum_{h\in [k]} \frac{\phi_{w_h}}{k}} \cdot \mathds{1}[X_u = \1] }.
%\end{align*}
\begin{align}\label{eq:k-trans-inf-3}
 \pi^{u_i \gets \1}_{k, v_j}(\1) &= 	\frac{\phi_{v_j}}{k} \cdot \frac{ \sum_{X \in \Omega(\mu): X_u = \1 \land X_v = \1} \mu(X)\prod_{w \in V \setminus \{u,v\}:X_w = \1} \varphi_w}{\sum_{X \in \Omega(\mu): X_u = \1} \mu(X)\prod_{w\in V\setminus \{u\}:X_w = \1} \varphi_w}\notag\\
 &= \frac{\phi_{v_j}}{\sum_{h \in [k]} \phi_{v_h}} \cdot \frac{ \sum_{X \in \Omega(\mu): X_u = \1 \land X_v = \1} \mu(X)\prod_{w \in V: X_w = \1} \varphi_w}{\sum_{X \in \Omega(\mu): X_u = \1} \mu(X)\prod_{w\in V:X_w = \1}\varphi_w} = \frac{\phi_{v_j}}{\sum_{h \in [k]} \phi_{v_h}} \cdot \pi^{u \gets \1}_v(\1),
\end{align}
where the last equation holds due to the definition of $\pi = \mu^{(\*\varphi)}$. %and the local fields $\*\varphi$ is defined in~\eqref{eq:varphi}.
%Recall that $\*\varphi \in \mathds{R}_{\geq 0}^{V}$ is a vector of the local fields defined in \eqref{eq:varphi} and $\pi = \mu^{(\*\varphi)}$, so it holds that
%\begin{align} \label{eq:k-trans-inf-3}
%  \pi^{u_i \gets \1}_{k, v_j}(\1)
%  &= \frac{\phi_{v_j}}{\sum_{h \in [k]} \phi_{v_h}} \cdot \frac{\Pr[X \sim \pi]{X_u = \1 \land X_v = \1}}{\Pr[X \sim \pi]{X_u = \1}}
%   = \frac{\phi_{v_j}}{\sum_{h \in [k]} \phi_{v_h}} \cdot \pi^{u \gets \1}_v(\1).
%\end{align}
%

Next, we calculate
\begin{align*}
  \pi^{u_i \gets \0}_{k, v_j}(\1) &= \frac{\Pr[Y \sim \pi_k]{Y_{u_i} = \0 \land Y_{v_j} = \1}}{\Pr[Y \sim \pi_k]{Y_{u_i} = \0}}.
\end{align*}
Let $\mathds{1}[\cdot]$ denote the indicator random variable.
By the definition of $\pi_k$, we have the following equation
\begin{align*} 
  &\Pr[Y\sim \pi_k]{Y_{u_i} = \0 \land Y_{v_j} = \1}\\ 
  =\,& \frac{\phi_{v_j}}{Z_{\*\phi}k}\sum_{\substack{X \in \Omega(\mu):\\ X_v = \1}} \mu(X)   \tp{\mathds{1}[X_u = \1]\tp{\sum_{h \in [k]\setminus \{i\}} \frac{\phi_{u_h}}{k}} + \mathds{1}[X_u = \0]}  \prod_{\substack{w \in V\setminus\{u, v\}: \\ X_w = \1}}\tp{\sum_{h \in [k]} \frac{\phi_{w_h}}{k}}\\
   =\,& \frac{\phi_{v_j}}{Z_{\*\phi}k}\sum_{\substack{X \in \Omega(\mu):\\ X_v = \1}} \mu(X)  \tp{\mathds{1}[X_u = \1]\varphi'_u + \mathds{1}[X_u = \0]} \prod_{\substack{w \in V\setminus\{u, v\}: \\ X_w = \1}}\varphi'_w\\
   =\,& \frac{\phi_{v_j}}{Z_{\*\phi}k}\sum_{\substack{X \in \Omega(\mu):\\ X_v = \1}} \mu(X) \prod_{\substack{w \in V\setminus\{v\}: \\ X_w = \1}}\varphi'_w
\end{align*}
where the vector $\*\varphi'\in \mathds{R}_{> 0}$ is the local fields such that
\begin{align} \label{eq:varphi-prime}
  \forall w \in V, \quad \varphi'_w = \begin{cases}
    \frac{1}{k} \sum_{h \in [k]} \phi_{w_h} & \text{ if } w \neq u \\
    \frac{1}{k} \sum_{h \in [k]\setminus \{i\}} \phi_{w_h} & \text{ if } w = u.
  \end{cases}
\end{align}
Similarly, it holds that
\begin{align*} 
  \Pr[Y\sim \pi_k]{Y_{u_i} = \0} &= \frac{1}{Z_{\*\phi}}\sum_{X \in \Omega(\mu)} \mu(X)  \tp{\mathds{1}[X_u = \1]\varphi'_u + \mathds{1}[X_u = \0]}  \prod_{\substack{w \in V\setminus\{u\}:\\ X_w = \1}}\varphi'_w = \frac{1}{Z_{\*\phi}}\sum_{X \in \Omega(\mu)} \mu(X)  \prod_{\substack{w \in V:\\ X_w = \1}}\varphi'_w.
\end{align*}
By a similar calculation to~\eqref{eq:k-trans-inf-3}, we can verify that 
\begin{align} \label{eq:k-trans-inf-4}
  \pi^{u_i \gets \0}_{k, v_j}(\1) &= \frac{\phi_{v_j}}{k} \cdot \frac{\sum_{X \in \Omega(\mu):X_v = \1} \mu(X)   \prod_{w \in V\setminus\{v\}:  X_w = \1}\varphi'_w}{\sum_{X \in \Omega(\mu)} \mu(X) \prod_{w \in V: X_w = \1}\varphi'_w} \notag\\
  &=\frac{\phi_{v_j}}{\sum_{h \in [k]} \phi_{v_h}} \cdot \frac{\sum_{X \in \Omega(\mu):X_v = \1} \mu(X)   \prod_{w \in V:  X_w = \1}\varphi'_w}{\sum_{X \in \Omega(\mu)} \mu(X) \prod_{w \in V: X_w = \1}\varphi'_w} = \frac{\phi_{v_j}}{\sum_{h \in [k]} \phi_{v_h}} \cdot \nu_v(\1),
\end{align}
where $\nu = \mu^{(\*\varphi')}$. 
Combining \eqref{eq:k-trans-inf-3} and \eqref{eq:k-trans-inf-4}, we know that
\begin{align} \label{eq:k-trans-inf-5}
  \abs{\pi^{u_i \gets \1}_{k, v_j}(\1) - \pi^{u_i \gets \0}_{v_j}(\1)} &= \frac{\phi_{v_i}}{\sum_{h\in [k]}\phi_{v_h}} \cdot  \abs{\pi^{u \gets \1}_v(\1) - \nu_v(\1)}.
\end{align}
In order to bound $\abs{\pi^{u \gets \1}_v(\1) - \nu_v(\1)}$, we construct the following coupling between $\pi^{u \gets \1}_v$ and $\nu_v$:
\begin{itemize}
\item sample a random value $c \in \{\0, \1\}$ according to the distribution $\nu_u$;
\item sample $c_v, c_v'$ jointly according to the optimal coupling between $\pi^{u \gets \1}_v$ and $\nu_v^{u \gets c}$.
\end{itemize}
Note that $\abs{\pi^{u \gets \1}_v(\1) - \nu_v(\1)}$ is the total variation distance between $\pi^{u \gets \1}_v$ and $\nu_v$. By the coupling inequality, we have
\begin{align*}
 \abs{\pi^{u \gets \1}_v(\1) - \nu_v(\1)}
  &\leq \Pr{c_v \not= c_v'} = \sum_{c \in \Omega(\nu_u)} \nu_u(c) \cdot \DTV{\pi^{u \gets \1}_v}{\nu^{u\gets c}_v}.
 % \nonumber &\overset{(\star)}{=} \sum_{c \in \Omega_{\pi'_u}} \pi'_u(c) \cdot \DTV{\pi^{u \gets \1}_v}{\pi^{u\gets c}_v}.
  %\label{eq:k-trans-inf-6} &\leq \max_{i, j \in \Omega(\pi_u)} \DTV{\pi^{u \gets i}_v}{\pi^{u \gets j}_v} = \abs{\Ima_{\pi}(u, v)},
\end{align*}
Recall $k \geq 2$, thus $\*\phi,\*\varphi$ and $\*\varphi'$ are all positive local fields and $\{\0,\1\}  = \Omega(\pi_{k, u_i})$.
It holds that $\Omega(\mu_u) = \Omega(\pi_u) = \Omega(\nu_u)=\{\0,\1\}$.
Note that  $\*\varphi$ and $\*\varphi'$ only differs at $u$. It is straightforward to verify that for any $c \in \Omega(\nu_u)$, $\pi^{u\gets c}_v = \nu^{u \gets c}_v$. We have
\begin{align*}
\abs{\pi^{u \gets \1}_v(\1) - \nu_v(\1)} \leq 	 \sum_{c \in \Omega({\nu_u})} \nu_u(c) \cdot \DTV{\pi^{u \gets \1}_v}{\pi^{u\gets c}_v} \leq  \DTV{\pi^{u \gets \0}_v}{\pi^{u \gets \1}_v} = \abs{\Ima_{\pi}(u, v)}.
\end{align*}
Combining above inequality with \eqref{eq:k-trans-inf-5} proves \eqref{eq:k-trans-inf-target-1}.

\subsection{Verification of \eqref{eq:k-trans-inf-target-2}}
Recall that in this case, we have $u = v$ and $i \neq j$.
Without loss of generality, we assume $\Omega(\pi_{k, u_i}) = \{\0, \1\}$, since otherwise \eqref{eq:k-trans-inf-target-2} holds trivially.
In this case, we have
\begin{align} \label{eq:k-trans-inf-7}
  \abs{\Ima_{\pi_k}(u_i, v_j)} &= \abs{\pi^{u_i \gets \1}_{k, u_j}(\1) - \pi^{u_i \gets \0}_{k, u_j}(\1)} = \pi^{u_i \gets \0}_{k, u_j}(\1),
\end{align}
where the last equation holds because $\pi^{u_i \gets \1}_{k, u_j}(\1) = 0$.
By the definition of the conditional distribution,
\begin{align*}
  \pi^{u_i \gets \0}_{k, u_j}(\1) &= \frac{\Pr[Y \sim \pi_k]{Y_{u_i} = \0 \land Y_{u_j} = \1}}{\Pr[Y \sim \pi_k]{Y_{u_i} = \0}}.
\end{align*}
By the definition of $\pi_k$, 
%we have
%\begin{align*}
%  \Pr[Y \sim \pi_k]{Y_{u_i} = \0 \land Y_{u_j} = \1} = \Pr[Y \sim \pi_k]{Y_{u_j} = \1}
%  = \frac{\phi_{u_j}}{Z_{\*\phi}k} \sum_{\substack{X \in \Omega(\mu):\\ X_u = 1}} \mu(X) \prod_{\substack{w \in V\setminus \{u\}:\\ X_w = \1}}\tp{\sum_{h \in [k]}\frac{\phi_{w_h}}{k}}.
%\end{align*}
%Similarly, 
we have
\begin{align*}
  \Pr[Y \sim \pi_k]{Y_{u_i} = -1}
  &= \frac{1}{Z_{\*\phi}} \sum_{X \in \Omega(\mu)} \mu(X)\tp{\mathds{1}[X_u = \1]\tp{\sum_{h \in [k]\setminus \{i\}}\frac{\phi_{u_h}}{k}} + \mathds{1}[X_u = \0]} \prod_{\substack{w \in V\setminus \{u\}: \\ X_w = \1}}\tp{\sum_{h \in [k]}\frac{\phi_{w_h}}{k}} \\
  &\geq \frac{1}{Z_{\*\phi}} \sum_{X \in \Omega(\mu)} \mu(X)  \mathds{1}[X_u = \1]\tp{\sum_{h \in [k]\setminus \{i\}}\frac{\phi_{u_h}}{k}}  \prod_{\substack{w \in V\setminus \{u\}:\\ X_w = \1}}\tp{\sum_{h \in [k]}\frac{\phi_{w_h}}{k}}\\
  &= \frac{1}{Z_{\*\phi}}  \tp{\sum_{h \in [k]\setminus \{i\}}\frac{\phi_{u_h}}{k}} \sum_{\substack{X \in \Omega(\mu):\\ X_u = 1}} \mu(X) \prod_{\substack{w \in V\setminus \{u\}:\\ X_w = \1}}\tp{\sum_{h \in [k]}\frac{\phi_{w_h}}{k}}\\
    &= \frac{1}{Z_{\*\phi}}  \tp{\frac{\sum_{h \in [k] \setminus \{i\}} \phi_{u_h}}{\phi_{u_j}} } \sum_{\substack{X \in \Omega(\mu):\\ X_u = 1}}  \frac{\mu(X)\phi_{u
    _j}}{k}\prod_{\substack{w \in V\setminus \{u\}:\\ X_w = \1}}\tp{\sum_{h \in [k]}\frac{\phi_{w_h}}{k}}\\
  &= \frac{\sum_{h \in [k] \setminus \{i\}} \phi_{u_h}}{\phi_{u_j}} \Pr[Y \sim \pi_k]{ Y_{u_j} = \1} = \frac{\sum_{h \in [k] \setminus \{i\}} \phi_{u_h}}{\phi_{u_j}} \Pr[Y \sim \pi_k]{Y_{u_i} = \0 \land Y_{u_j} = \1},
\end{align*}
where the last equation holds because if $Y_{u_j} = \1$, then $Y_{u_{i}} = \0$.
%Note that $k \geq 2$ and $\*\phi \in \mathds{R}_{> 0}^{V \times [k]}$, it holds that $\sum_{h \in [k]\setminus \{i\}}\phi_{u_h} > 0$.
This implies that  
\begin{align} \label{eq:k-trans-inf-8}
  \pi^{u_i\gets \0}_{k, u_j}(\1) = \frac{\Pr[Y \sim \pi_k]{Y_{u_i} = \0 \land Y_{u_j} = \1}}{\Pr[Y \sim \pi_k]{Y_{u_i} = \0}} \leq \frac{\phi_{u_j}}{\sum_{h \in [k] \setminus \{i\}} \phi_{u_h}}.
\end{align}
%Now, combining \eqref{eq:k-trans-inf-7} and \eqref{eq:k-trans-inf-8}, we finish the verification of \eqref{eq:k-trans-inf-target-2}.

%\section{Blcok Factorization $\Rightarrow$ Field Dynamics: Concentration Argument}
\section{Magnetized Block Factorization of Entropy} \label{sec:FBF}
In this section, we prove \Cref{lem:FBF}.
Let $\mu$ be a distribution over $\{\0, \1\}^V$.
Let $\theta \in (0, 1)$ be a real number, recall that $\pi = \mu^{\tp{\theta}}$ is a distribution and $Z_\pi$ is the normalization factor defined in \eqref{eq-redef-pi}.
Recall that we use $\Bin(V, 1 - \theta)$ to denote the distribution of subset $\+R \subseteq V$ that is randomly generated by including each $v \in V$ into $R$ independently with probability $1 - \theta$.
%Specifically, for every $\Lambda \subseteq V$,
%$\Pr[R \sim \Bin(V, 1 - \theta)]{R = \Lambda} = (1 - \theta)^{\abs{\Lambda}}\theta^{\abs{V} - \abs{\Lambda}}.$
Let $\eta > 0$.
Suppose $\mu_k$ satisfies the $\lceil \theta kn \rceil$-uniform block factorization with constant $C$ for all integers $k \geq k_0$.
To  prove \Cref{lem:FBF}, we need to show that for any $f: \{\0, \1\}^V \to \mathds{R}_{\geq 0}$, it holds that
 \begin{align*}
    \Ent[\mu]{f} &\leq C \cdot\frac{Z_\pi}{\theta^{\abs{V}}} \E[R\sim \Bin(V, 1 - \theta)]{\pi_R\tp{\mathds{1}_R} \cdot \Ent[\pi^{\mathds{1}_R}]{f}},
 \end{align*}
 where we assume $\pi_R\tp{\mathds{1}_R} \cdot \Ent[\pi^{\mathds{1}_R}]{f} = 0$ if $\pi_R\tp{\mathds{1}_R} = 0$.

Fix an integer $k \geq k_0$ and a function $f:\{\0,\1\}^V \to \mathds{R}_{\geq 0}$.
We define a new function $f^k: \{\0, \1\}^{V_k} \to \mathds{R}_{\geq 0}$ as
\begin{align*}
  \forall \sigma \in \{\0, \1\}^{V_k}, \quad f^k(\sigma) = f(\sigma^\star).
\end{align*}
Recall that 
\begin{align*}
\forall v \in V,\quad \sigma^\star_v = \begin{cases}
 \1 & \exists i \in [k] \text{ s.t. } \sigma_{(v,i)} = \1;\\
 \0 & \forall i \in [k], \sigma_{(v,i)} = \0.
 \end{cases}
\end{align*}
Since $\mu_k$ satisfies the $\lceil \theta kn \rceil$-uniform block factorization with parameter $C$, by definition, we have
\begin{align*}
\Ent[\mu_k]{f^{k}} \leq \frac{C}{\binom{nk}{\lceil \theta kn \rceil}} \sum_{S \in \binom{V_k}{\lceil \theta kn \rceil}}\mu_k\left[\Ent[S]{f^{k}}\right].
\end{align*}

Note that above inequality holds for all $k \geq k_0$.
\Cref{lem:FBF} is a straightforward consequence of the following two results
\begin{align}
\Ent[\mu_k]{f^{k}} &= \Ent[\mu]{f} \label{eq-limit-1};\\
\lim_{k \to \infty} \frac{C}{\binom{nk}{\lceil \theta kn \rceil}} \sum_{S \in \binom{V_k}{\lceil \theta kn \rceil}}\mu_k\left[\Ent[S]{f^{k}}\right]  &= C \cdot\frac{Z_\pi}{\theta^{\abs{V}}} \E[R\sim \Bin(V, 1 - \theta)]{\pi_R\tp{\mathds{1}_R} \cdot \Ent[\pi^{\mathds{1}_R}]{f}}.\label{eq-limit-2}
\end{align}
In the rest of section, we verify~\eqref{eq-limit-1} and~\eqref{eq-limit-2}.

\subsection{Verification of~\eqref{eq-limit-1}}
By definition, it holds that
\begin{align*}
\Ent[\mu_k]{f^k} = \sum_{\sigma \in \Omega(\mu_k)}\mu_k(\sigma) f^k(\sigma)\log f^k(\sigma)	- \sum_{\sigma \in \Omega(\mu_k)}\mu_k (\sigma)f^k(\sigma) \log \sum_{\tau \in \Omega(\mu_k)}\mu_k(\tau) f^k(\tau).
\end{align*}
Note that
\begin{align*}
 \sum_{\sigma \in \Omega(\mu_k)}\mu_k(\sigma) f^k(\sigma)\log f^k(\sigma) &= \sum_{\alpha \in \Omega(\mu)}\sum_{\sigma \in \Omega(\mu_k):\sigma^\star = \alpha}\mu_k(\sigma)f(\alpha)\log f(\alpha)= \sum_{\alpha \in \Omega(\mu)}\mu(\alpha)f(\alpha)\log f(\alpha);\\
 \sum_{\sigma \in \Omega(\mu_k)}\mu_k(\sigma)f^k(\sigma) &= \sum_{\alpha \in \Omega(\mu)} \sum_{\sigma \in \Omega(\mu_k):\sigma^\star = \alpha}\mu_k(\sigma)f(\alpha) = \sum_{\alpha \in \Omega(\mu)}\mu(\alpha)f(\alpha).
\end{align*}
This implies the equation $\Ent[\mu_k]{f^k} = \Ent[\mu]{f}$.

\subsection{Verification of~\eqref{eq-limit-2}}

We introduce some notations. Let $\*\varphi \in [0, 1]^V$ be a vector. For any subset $\Lambda \subseteq V$, we use $\*\varphi_\Lambda \in [0,1]^V$ denote the vector induced by restricting $\*\varphi$ on $\Lambda$.
Let $\mu^{(\*\varphi_\Lambda)}$ denote the distribution obtained by imposing local fields $\*\varphi_\Lambda$ on $\mu$, formally
\begin{align*}
\forall \sigma \in \{\0,\1\}^V,\quad \mu^{(\*\varphi_\Lambda)}(\sigma) = \frac{\mu(\sigma)}{Z(\*\varphi_\Lambda)}\prod_{\substack{v \in \Lambda:\\ \sigma_v = \1}}\*\varphi_\Lambda(v), \quad\text{where } Z(\*\varphi_\Lambda)\triangleq \sum_{\tau \in \{\0,\1\}^V}\mu(\tau)\prod_{\substack{v \in \Lambda:\\ \tau_v = \1}}\*\varphi_\Lambda(v). 	
\end{align*}
The above definition can be viewed as a generalization of \Cref{definition-local-fields}.
The local fields $\*\varphi_\Lambda$ can only be defined on a subset $\Lambda$. 
It is easy to see $\mu^{(\*\varphi_\Lambda)} = \mu^{(\*\varphi_\Lambda \cup \mathds{1}_{V \setminus \Lambda})}$.
Besides, we allow $\*\varphi_\Lambda(v) = 0$ for some vertices $v \in V$. The distribution $\mu^{(\*\varphi_\Lambda)}$ is well-defined if and only if $Z(\*\varphi_\Lambda) > 0$.

Fix an integers $k \geq 1$ and $1\leq \ell \leq kn$. 
%Recall that $\HyperGeo$ is the multivariate hypergeometric distribution defined in~\eqref{eq-hypergeo}. 
We have the following lemma. %For any configuration $\sigma \in \Omega(\mu_k)$, define 
%\begin{align*}
%F(\sigma) \triangleq \{v \in V \mid \exists\, i \in [k] \text{ s.t. } \sigma_{v_i} = \1 \}.	
%\end{align*}

\begin{lemma} \label{lem:PBF} % projected block fatorization
Let $C > 0$ and $k \geq 1$ and $1\leq \ell \leq kn$ be two integers. 
  %Let $C > 0$ be a real number and $k, \ell$ be integers that satisfying $k\geq 1$ and $1 \leq \ell \leq kn$.
  If $\mu_k$ satisfies the $\ell$-uniform block factorization with constant $C$, then for any $f: \{\0, \1\}^V \to \mathds{R}_{\geq 0}$, let
  \begin{align*}
  H_f(\ell,k) \triangleq \frac{1}{\binom{nk}{\ell}} \sum_{S \in \binom{V_k}{\ell}}\mu_k\left[\Ent[S]{f^{k}}\right],
  \end{align*}
it holds that 
%\begin{align*}
%H_f(\ell,k) = \sum_{\substack{R \subseteq V : \mu_R(\mathds{1}_R) > 0 \\ \*b: k\*b \in \Omega(\HyperGeo) }}  \HyperGeo(k\*b) \prod_{u \in R}(1-b_u) \cdot \Ent[\mu^{(\*b_{V \setminus R }),\mathds{1}_R}]{f} \sum_{\substack{\tau \in \Omega(\mu):\\ \tau_R = \mathds{1}_R}}\mu(\tau)\prod_{v \in \tau^{-1}(\1) \setminus R}b_v,	
%\end{align*}
\begin{align*}
  H_f(\ell,k) =  \sum_{\*b: k\*b \in \Omega(\HyperGeo)}\HyperGeo(k\*b) \sum_{\tau \in \Omega(\mu)} \mu(\tau) \sum_{\substack{R \subseteq \tau^{-1}(\1):\\ \forall v \in \tau^{-1}(\1)\setminus R,\; b_v > 0}}   \prod_{u \in R}(1-b_u)\prod_{v \in \tau^{-1}(\1) \setminus R}b_v \cdot \Ent[\mu^{(\*b_{V \setminus R }),\mathds{1}_R}]{f}.
\end{align*}
where  $\mathds{1}_R$ is the all-($\1$) configuration on $R$ and $\tau^{-1}(\1)= \{v \in V \mid \tau_v = \1\}$, the distribution $\mu^{(\*b_{V \setminus R }),\mathds{1}_R}$ is obtained from $\mu^{(\*b_{V \setminus R })}$ conditional on $\mathds{1}_R$, and $ \HyperGeo(\cdot)$ is the multivariate hypergeometric distribution in~\eqref{eq-hypergeo}.
%  \begin{align*} 
%    \Ent[\mu]{f}
%    &\leq C \cdot H_f(\ell,k) \triangleq C \cdot \sum_{\substack{R \subseteq V : \mu_R(\mathds{1}_R) > 0 \\ \*b: k\*b \in \Omega(\HyperGeo) }}\prod_{u \in R}(1-b_u) \cdot \Ent[\mu^{(\*b_{V \setminus R }),\mathds{1}_R}]{f} \sum_{\substack{\tau \in \Omega(\mu):\\ \tau_R = \mathds{1}_R}}\mu(\tau)\prod_{v \in \tau^{-1}(\1) \setminus R}b_v,
%    %\sum_{\*b: k\*b \in \Omega(\HyperGeo)} \HyperGeo(k\*b) \sum_{H\subseteq V} \mathds{1}_{\+D}[\mu^{\tp{\*b^H}}] \cdot Z\tp{\*b^H} \cdot \mu^{\tp{\*b^H}}_H(\mathds{1}_H) \cdot \Ent[\mu^{\tp{\*b^H}, \mathds{1}_H}]{f}.
%  \end{align*}
\end{lemma}

The above lemma relates $H_f(\ell,k)$ to the multivariate hypergeometric distribution. 
By the concentration property, we have the following lemma.
%Given a function $f: \{\0, \1\}^V \to \mathds{R}_{\geq 0}$, define the function $G_f$ as 
%\begin{align*}
%G_f(\ell, k) \triangleq \sum_{\sigma \in \Omega(\mu_k)}\mu_k(\sigma) \sum_{\substack{R \subseteq F(\sigma)\\ \*b: k\*b \in \Omega(\HyperGeo)}} \prod_{u \in R}(1-b_u)\prod_{v \in F(\sigma) \setminus R}b_v \cdot \Ent[\mu^{(\*b_{V \setminus R }),\mathds{1}_R}]{f}.	
%\end{align*}\The following lemma analyzes $H_f(\ell, k)$ when $k \to \infty$.
%The RHS is , which can be bounded as in the following lemma.
\begin{lemma} \label{lem:LBF} % limit of block factorization
  Let $\theta \in (0, 1)$ and $\pi = \mu^{(\theta)}$.
  Let  $f: \{\0, \1\}^V \to \mathds{R}_{\geq 0}$ be a function. It holds that 
  \begin{align*}
    \lim_{k \to \infty} H_f\tp{ \ctp{\theta n k}, k}
    =\frac{Z_\pi}{\theta^{\abs{V}}} \E[R\sim \Bin(V, 1 - \theta)]{\pi_R\tp{\mathds{1}_R} \cdot \Ent[\pi^{\mathds{1}_R}]{f}}.
    %\sum_{R \subseteq V: \mu_R(\mathds{1}_R)> 0}(1-\theta)^{\abs{R}} \cdot \Ent[\pi^{\mathds{1}_R}]{f} \sum_{\substack{\tau \in \Omega(\mu):\\ \tau_R = \mathds{1}_R}}\mu(\tau)\theta^{\norm{\tau}_+ - |R|},
  \end{align*}
%where $\norm{\tau}_+ = |\{v \in V \mid \tau_v = \1\}|$.
\end{lemma}

Equation~\eqref{eq-limit-2} can be verified by combining~\Cref{lem:PBF} and \Cref{lem:LBF}.
We verify these two lemmas in the rest of this section.

\subsection{Proof of \Cref{lem:PBF}}
%Fix an integer $k \geq 1$ and a function $f: \{\0,\1\}^V \to \mathds{R}_{\geq 0}$.
%We define a new function $f^k: \{\0, \1\}^{V_k} \to \mathds{R}_{\geq 0}$ as
%\begin{align*}
%  \forall \sigma \in \{\0, \1\}^{V_k}, \quad f^k(\sigma) = f(\sigma^\star).
%\end{align*}
%Recall that 
%\begin{align*}
%\forall v \in V,\quad \sigma^\star_v = \begin{cases}
% \1 & \exists i \in [k] \text{ s.t. } \sigma_{(v,i)} = \1;\\
% \0 & \forall i \in [k], \sigma_{(v,i)} = \0.
% \end{cases}
%\end{align*}
%Since  $\mu_k$ satisfies the $\ell$-uniform block factorization with constant $C$, we have
For any subset $\Lambda \subseteq V_k$, we use $\mu_{k,\Lambda}$ to denote the marginal distribution on $\Lambda$ induced from $\mu_k$.
By definition, we have
\begin{align*}
  H_f(\ell,k)
  &= \frac{1}{\binom{kn}{\ell}} \cdot \sum_{S \in \binom{V_k}{\ell}} \sum_{\tau \in \Omega(\mu_{k, V_k\setminus S})} \mu_{k, V_k\setminus S}(\tau) \cdot \Ent[\mu^{\tau}_k]{f^k}\\
  & = \frac{1}{\binom{kn}{\ell}} \cdot \sum_{\sigma \in \Omega(\mu_k)} \mu_k(\sigma)\sum_{S \in \binom{V_k}{\ell}} \Ent[\mu^{\rho}_k]{f^k}, \quad\text{where } \rho = \rho(\sigma, S) \triangleq \sigma_{V_k \setminus S}.
\end{align*}
Recall that $V_k = V \times [k]$. For each $v \in V$, define $C_v = \{v_1,v_2,\ldots,v_k\}$. 
For configuration $\sigma \in \Omega(\mu_k)$, define 
\begin{align*}
F(\sigma) \triangleq \{v \in V \mid \exists\, i \in [k] \text{ s.t. } \sigma_{v_i} = \1 \}.	
\end{align*}
Similarly, for partial configuration $\rho = \sigma_{V_k \setminus S}$, this definition could be extended as
\begin{align*}
F(\rho)  \triangleq \{v \in V \mid \exists\, i \in [k] \text{ s.t. } v_i \in V_k \setminus S \land \sigma_{v_i} = \1 \}.	
\end{align*}
Next, we define local fields $\*\phi_\rho \in [0,1]^{V \setminus F(\rho)}$ such that
\begin{align*}
\forall v \in V \setminus F(\rho), \quad 	\phi_\rho(v) \triangleq \frac{|S \cap C_v|}{k}.
\end{align*}
%We have the following claim.
\begin{claim}
\label{claim-Ent}
It holds that  $\Ent[\mu^\rho_k]{f^k} = \Ent[\mu^{(\*\phi_\rho), \mathds{1}_{F(\rho)}}]{f}$.
\end{claim}
By~\Cref{claim-Ent}, we have 
\begin{align}
\label{eq-Enf-rho}
H_f(\ell,k)  = \frac{1}{\binom{kn}{\ell}} \cdot \sum_{\sigma \in \Omega(\mu_k)} \mu_k(\sigma)\sum_{S \in \binom{V_k}{\ell}} \Ent[\mu^{(\*\phi_\rho), \mathds{1}_{F(\rho)}}]{f}, \quad\text{where } \rho = \rho(\sigma, S) \triangleq \sigma_{V_k \setminus S}.
\end{align}
Note that $S$ is sampled from $\binom{V_k}{\ell}$ uniformly at random. Let $\*a = (a_v)_{v \in V}$, where $a_v = |S \cap C_v|$. Vector $\*a$ follows the distribution multivariate hypergeometric distribution $\HyperGeo$. We can use the following procedure to generate the random set $S$: sample $\*a \sim \HyperGeo$; sample $S_v \in \binom{C_v}{a_v}$ uniformly at random and let $S = \cup_{v \in V}S_v$.
%Recall $F(\sigma) \triangleq \{v \in V \mid \exists\, i \in [k] \text{ s.t. } \sigma_{v_i} = \1 \}$.
Note that $F(\rho) \subseteq F(\sigma)$. Let $\*b  = \*a /k$, so that $\phi_\rho(v) = b_v$ for all $v \in V \setminus F(\rho)$.
%Denote $R = F(\rho)$.
%Equation~\eqref{eq-Enf-rho} can be rewritten as
By letting $R = F(\rho)$,
we rewrite equation~\eqref{eq-Enf-rho} as 
%\begin{align*}
%H_f(\ell,k)  =  \sum_{\sigma \in \Omega(\mu_k)} \mu_k(\sigma)	\sum_{\substack{R \subseteq F(\sigma)\\ \*b: k\*b \in \Omega(\HyperGeo)}}  \HyperGeo(k\*b) \prod_{u \in R}(1-b_u)\prod_{v \in F(\sigma) \setminus R}b_v \cdot \Ent[\mu^{(\*b_{V \setminus R }),\mathds{1}_R}]{f}.
%\end{align*}
\begin{align*}
H_f(\ell,k)  =  \sum_{\*b: k\*b \in \Omega(\HyperGeo)}\HyperGeo(k\*b) \sum_{\sigma \in \Omega(\mu_k)} \mu_k(\sigma)	\sum_{\substack{R \subseteq F(\sigma):\\ \forall v \in F(\sigma)\setminus R,\; b_v > 0}}   \prod_{u \in R}(1-b_u)\prod_{v \in F(\sigma) \setminus R}b_v \cdot \Ent[\mu^{(\*b_{V \setminus R }),\mathds{1}_R}]{f}.
\end{align*}
Here, vector $k\*b$ determines the size of each $S_v$.
Then, for each $R\subseteq F(\sigma)$, the following part of above equation specifies the probability that $F(\rho)=R$:
\begin{itemize}
\item when $u \in R$, the $\1$-vertex in $C_u$ should not be selected into $S_u$, this happens with probability $1 - b_u$;
\item when $u \in F(\sigma) \setminus R$, the $\1$-vertex in $C_u$ must be selected into $S_u$, so it happens with probability $b_u$.
\end{itemize}
If there exists $v \in F(\sigma) \setminus R$ such that $b_v = 0$, then by above argument, we know that $\Pr{R = F(\rho)} = 0$.
Since it is no need to enumerate this kind of $R$, we add a constraint for $R$ which says $\forall v \in F(\sigma) \setminus R, \; b_v > 0$.
This constraint for enumerating $R$ also ensures that $\mu^{(\*b_{V \setminus R }),\mathds{1}_R}$ will always be well defined.

%Each subset $R$ enumerated in above inequality satisfies that there exists $\sigma \in \Omega(\mu_k)$ such that $R \subseteq F(\sigma)$, which is equivalent to $\mu_R(\mathds{1}_R) > 0$. We have
%\begin{align*}
%H_f(\ell,k) = \sum_{\substack{R \subseteq V : \\ \mu_R(\mathds{1}_R) > 0 }}\sum_{\*b: k\*b \in \Omega(\HyperGeo)} \HyperGeo(k\*b) \prod_{u \in R}(1-b_u) \cdot \Ent[\mu^{(\*b_{V \setminus R }),\mathds{1}_R}]{f} \sum_{\substack{\sigma \in \Omega(\mu_k):\\R \subseteq F(\sigma)}}\mu_k(\sigma)\prod_{v \in F(\sigma) \setminus R}b_v.
%%&=  C \cdot \sum_{\substack{R \subseteq V : \mu_R(\mathds{1}_R) > 0 }}\sum_{\*b: k\*b \in \Omega(\HyperGeo)}\prod_{u \in R}(1-b_u)\prod_{v \in F(\sigma) \setminus R}b_v \cdot \Ent[\mu^{(\*b_{V \setminus R }),\mathds{1}_R}]{f} \mu_R(\mathds{1}_R).
%\end{align*}

To generate a random configuration $\sigma \sim \mu_k$, one can sample $\tau \sim \mu$, then transform $\tau$ to $\sigma$ by \Cref{def:k-trans}.
It turns out that $F(\sigma)$ can be uniquely determined by $\tau$, because $\tau(v) = \1$ if and only if $v \in F(\sigma)$.
%Furthermore, $F(\sigma)$ actually uniquely determines a configuration $\tau \in \Omega(\mu)$.
%So, in turn we could enumerate $\tau \in \Omega(\mu)$ first, then we know that $F(\sigma) = \tau^{-1}(\1)$.
%Now, above equation could be written as
To enumerate all configurations $\sigma \in \Omega(\mu_k)$, we first enumerates all configurations $\tau \in \Omega(\mu)$, then enumerate all $\sigma$ such that $\sigma^\star = \tau$.
We have
\begin{align*}
  &H_f(\ell,k) \\
  &=  \sum_{\*b: k\*b \in \Omega(\HyperGeo)}\HyperGeo(k\*b) \sum_{\tau \in \Omega(\mu)} \sum_{\substack{R \subseteq \tau^{-1}(\1):\\ \forall v \in \tau^{-1}(\1)\setminus R,\; b_v > 0}}   \prod_{u \in R}(1-b_u)\prod_{v \in \tau^{-1}(\1) \setminus R}b_v \cdot \Ent[\mu^{(\*b_{V \setminus R }),\mathds{1}_R}]{f} \cdot \sum_{\substack{\sigma \in \Omega(\mu_k):\\ \sigma^\star = \tau}} \mu_k(\sigma)	
\end{align*}
%Observe that 
%\begin{align*}
%\sum_{\substack{\sigma \in \Omega(\mu_k):\\R \subseteq F(\sigma)}}\mu_k(\sigma)\prod_{v \in F(\sigma) \setminus R}b_v = \sum_{\substack{\tau \in \Omega(\mu):\\ \tau_R = \mathds{1}_R}}\mu(\tau)\prod_{v \in \tau^{-1}(\1) \setminus R}b_v.
%\end{align*}
By definition, it is easy to verify that
\begin{align*}
  \mu(\tau) &= \sum_{\substack{\sigma \in \Omega(\mu_k):\\ \sigma^\star = \tau}} \mu_k(\sigma).
\end{align*}
This implies that
%\begin{align*}
%H_f(\ell,k) = \sum_{\substack{R \subseteq V :\\  \mu_R(\mathds{1}_R) > 0 }}\sum_{\*b: k\*b \in \Omega(\HyperGeo)}  \HyperGeo(k\*b) \prod_{u \in R}(1-b_u) \cdot \Ent[\mu^{(\*b_{V \setminus R }),\mathds{1}_R}]{f} \sum_{\substack{\tau \in \Omega(\mu):\\ \tau_R = \mathds{1}_R}}\mu(\tau)\prod_{v \in \tau^{-1}(\1) \setminus R}b_v.
%\end{align*}
\begin{align*}
  H_f(\ell,k) =  \sum_{\*b: k\*b \in \Omega(\HyperGeo)}\HyperGeo(k\*b) \sum_{\tau \in \Omega(\mu)} \mu(\tau) \sum_{\substack{R \subseteq \tau^{-1}(\1):\\ \forall v \in \tau^{-1}(\1)\setminus R,\; b_v > 0}}   \prod_{u \in R}(1-b_u)\prod_{v \in \tau^{-1}(\1) \setminus R}b_v \cdot \Ent[\mu^{(\*b_{V \setminus R }),\mathds{1}_R}]{f}.
\end{align*}
This proves \Cref{lem:PBF}.

\begin{proof}[Proof of \Cref{claim-Ent}]

By \cite[Lemma 6.2]{chen2021rapid}, it is straightforward to verify
\begin{align}
\label{eq-CFYZ}
\forall \xi \in \Omega(\mu),\quad
\mu^{(\*\phi_\rho),\mathds{1}_{F(\rho)}}(\xi) = \sum_{{\tau \in \Omega(\mu_k):\tau^\star = \xi}}\mu^\rho_k(\tau).
\end{align}
By definition, it holds that
\begin{align*}
\Ent[\mu^\rho_k]{f^k} = \sum_{\sigma \in \Omega(\mu_k)}\mu^\rho_k(\sigma) f^k(\sigma)\log f^k(\sigma)	- \sum_{\sigma \in \Omega(\mu_k)}\mu^\rho_k (\sigma)f^k(\sigma) \log \sum_{\tau \in \Omega(\mu_k)}\mu^\rho_k(\tau) f^k(\tau).
\end{align*}
In above equation, we enumerate all $\sigma \in \Omega(\mu_k)$. Since $\Omega(\mu_k^\rho)\subseteq \Omega(\mu_k)$, the above equation is correct.
By~\eqref{eq-CFYZ}, it holds that 
\begin{align*}
 &\sum_{\sigma \in \Omega(\mu_k)}\mu^\rho_k(\sigma) f^k(\sigma)\log f^k(\sigma) = 
 \sum_{\alpha \in \Omega(\mu)}\sum_{\sigma \in \Omega(\mu_k):\sigma^\star = \alpha} \mu^\rho_k(\sigma) f^k(\sigma)\log f^k(\sigma) \\
 =\,&  \sum_{\alpha \in \Omega(\mu)}f(\alpha)\log f(\alpha) \sum_{\sigma \in \Omega(\mu_k):\sigma^\star = \alpha} \mu^\rho_k(\sigma)=\sum_{\alpha \in \Omega(\mu)} \mu^{(\*\phi_\rho),\mathds{1}_{F(\rho)}}(\alpha) f(\alpha) \log f(\alpha)	
\end{align*}
%$ \sum_{\sigma \in \Omega(\mu_k)}\mu^\rho_k(\sigma) f^k(\sigma)\log f^k(\sigma) = \sum_{\alpha \in \Omega(\mu)} \mu^{(\*\phi_\rho),\mathds{1}_{F(\rho)}}(\alpha) f(\alpha) \log f(\alpha) $  
Similarly, it holds that
\begin{align*}
\sum_{\sigma \in \Omega(\mu_k)}\mu^\rho_k (\sigma)f^k(\sigma) = \sum_{\alpha \in \Omega(\mu)}\mu^{(\*\phi_\rho),\mathds{1}_{F(\rho)}}(\alpha) f(\alpha).	
\end{align*}
Note that $\Omega(\mu^{(\*\phi_\rho),\mathds{1}_{F(\rho)}}) \subseteq \Omega(\mu)$. This implies $\Ent[\mu^\rho_k]{f^k} = \Ent[\mu^{(\*\phi_\rho), \mathds{1}_{F(\rho)}}]{f}$.
\end{proof}

\subsection{Proof of \Cref{lem:LBF}}
Fix the distribution $\mu$, parameters $\theta \in (0,1)$,  and a function $f:\{\0,\1\}^V \to \mathds{R}_{\geq 0}$.
We define the following function $H: [0,1]^V \to \mathds{R}$:
%\begin{align*}
%H(\*b) = \sum_{\substack{R \subseteq V : \mu_R(\mathds{1}_R) > 0}}\prod_{u \in R}(1-b_u) \cdot \Ent[\mu^{(\*b_{V \setminus R }),\mathds{1}_R}]{f} \sum_{\substack{\tau \in \Omega(\mu):\\ \tau_R = \mathds{1}_R}}\mu(\tau)\prod_{v \in \tau^{-1}(\1) \setminus R}b_v.
%\end{align*}
\begin{align*}
  H(\*b) =  \sum_{\tau \in \Omega(\mu)} \mu(\tau) \sum_{\substack{R \subseteq \tau^{-1}(\1):\\ \forall v \in \tau^{-1}(\1)\setminus R,\; b_v > 0}}   \prod_{u \in R}(1-b_u)\prod_{v \in \tau^{-1}(\1) \setminus R}b_v \cdot \Ent[\mu^{(\*b_{V \setminus R }),\mathds{1}_R}]{f}.
\end{align*}
Let $\ell = \ctp{\theta n k}$. We have
\begin{align*}
H_f(\ell, k) = \E[k\*b \sim \HyperGeo]{H(\*b)}.	
\end{align*}
We define a bad event $\+B_{\delta, k}$ for vector $\*b$ such that 
\begin{align*}
\+B_{\delta, k}: \exists v \in V, \abs{b_v - \theta} \geq \delta. 
\end{align*}
\begin{claim}
\label{claim-concentration}
For any $\epsilon > 0$, any $\delta > 0$, there exists $K > 0$ such that for all $k \geq K$, $\Pr[k\*b \sim \HyperGeo]{\+B_{\delta, k}} \leq \epsilon$.
\end{claim}

Let $\*\theta$ denote the constant vector with value $\theta$. 
Note that $H(\*b)$ is a continuous function in $(0,1)^V$.
Given any $\epsilon > 0$, we can find $0 < \delta = \delta(\epsilon) < \frac{\theta}{2}$ such that for any $\*b \in (0,1)^V$ with $\norm{\*b-\*\theta}_\infty < \delta$, 
\begin{align*}
\abs{H(\*b) - H(\*\theta)} \leq \frac{\epsilon}{2}.	
\end{align*}
Besides, for any vector $\*b \in [0,1]^V$, we have
\begin{align*}
H(\*b) \leq M = M(\mu,f),	
\end{align*}
where $M$ depends only on $\mu$ and $f$.
By \Cref{claim-concentration}, there exists $K > 0$ such that for all $k \geq K$, it holds that $\Pr[k\*b \sim \HyperGeo]{\+B_{\delta, k}} \leq \frac{\epsilon}{2\max\{H(\*\theta), M \}}$. Hence for all $k \geq K$, we have 
\begin{align*}
H_f(\ell, k) &= \E[k\*b \sim \HyperGeo]{H(\*b)} \geq 	\tp{1 - \Pr[k\*b \sim \HyperGeo]{\+B_{\delta, k}}}\tp{H(\*\theta) - \frac{\epsilon}{2}} \geq H(\*\theta) - \epsilon\\
H_f(\ell, k) &= \E[k\*b \sim \HyperGeo]{H(\*b)} \leq H(\*\theta) + \frac{\epsilon}{2} + M\Pr[k\*b \sim \HyperGeo]{\+B_{\delta, k}} \leq  H(\*\theta) + \epsilon.
\end{align*}
Let $\norm{\tau}_+ = |\{v \in V \mid \tau_v = \1\}|$.
This implies that 
%\begin{align*}
%   \lim_{k \to \infty} H_f\tp{ \ctp{\theta n k}, k} = H(\*\theta) &= \sum_{\substack{R \subseteq V : \mu_R(\mathds{1}_R) > 0}}(1-\theta)^{|R|} \cdot \Ent[\mu^{(\*\theta_{V \setminus R }),\mathds{1}_R}]{f} \sum_{\substack{\tau \in \Omega(\mu):\\ \tau_R = \mathds{1}_R}}\mu(\tau)\theta^{\norm{\tau}_+ - |R|} \\
%    &=\sum_{\substack{R \subseteq V : \mu_R(\mathds{1}_R) > 0}}(1-\theta)^{|R|} \cdot \Ent[\pi^{\mathds{1}_R}]{f} \sum_{\substack{\tau \in \Omega(\mu):\\ \tau_R = \mathds{1}_R}}\mu(\tau)\theta^{\norm{\tau}_+ - |R|},
%\end{align*}
\begin{align*}
  \lim_{k \to \infty} H_f\tp{ \ctp{\theta n k}, k} = H(\*\theta)
  &=  \sum_{\tau \in \Omega(\mu)} \mu(\tau) \sum_{R \subseteq \tau^{-1}(\1)}  (1-\theta)^{\abs{R}} \theta^{\norm{\tau}_+ - \abs{R}} \cdot \Ent[\mu^{(\*\theta_{V \setminus R }),\mathds{1}_R}]{f} \\
  &=  \sum_{R\subseteq V: \mu_R(\mathds{1}_R) > 0} (1-\theta)^{\abs{R}} \theta^{- \abs{R}}\Ent[\pi^{\mathds{1}_R}]{f}\sum_{\substack{\tau \in \Omega(\mu):\\ \tau_R = \mathds{1}_R}} \mu(\tau) \theta^{\norm{\tau}_+},
\end{align*}
where the last equation holds due to $\pi = \mu^{(\theta)}$ and $\mu^{(\*\theta_{V \setminus R }),\mathds{1}_R} = \pi^{\mathds{1}_R}$, this is because  $\pi$ and $\mu^{(\*\theta_{V \setminus R })}$ disagree only at the local fields on $R$ and the configuration on $R$ is fixed by $\mathds{1}_R$. Furthermore, we have
\begin{align*}
 \lim_{k \to \infty} H_f\tp{ \ctp{\theta n k}, k} 	&= \frac{Z_\pi}{\theta^{|V|}} \sum_{R \subseteq V:\mu_R(\mathds{1}_R) > 0}(1-\theta)^R\theta^{|V|-|R|}\Ent[\pi^{\mathds{1}_R}]{f}\pi_R(\mathds{1}_R)\\
 &= \frac{Z_\pi}{\theta^{\abs{V}}} \E[R\sim \Bin(V, 1 - \theta)]{\pi_R\tp{\mathds{1}_R} \cdot \Ent[\pi^{\mathds{1}_R}]{f}},
\end{align*}
where we assume $\pi_R\tp{\mathds{1}_R} \cdot \Ent[\pi^{\mathds{1}_R}]{f} = 0$ if $\pi_R\tp{\mathds{1}_R} = 0$.

\begin{proof}[Proof of \Cref{claim-concentration}]
Recall that $n = |V|$ and $\ell = \lceil \theta kn \rceil$.
Observe that when $k \ge \frac{2}{\delta}$,
  \begin{align}
  \label{eq:bad-1}
    \Pr[k\*b \sim {\HyperGeo[\ell]} ]{\+B_{\delta,k}} 
    &= \Pr[k\*b \sim {\HyperGeo[\ell_k]} ]{\exists v \in V, \abs{b_v-\theta} \ge \delta} \notag\\
    &\le \sum_{v \in V} \Pr[k\*b \sim {\HyperGeo[\ell]}]{\abs{b_v-\theta} \ge \delta}\notag\\
    &\le \sum_{v \in V} \Pr[k\*b \sim {\HyperGeo[\ell]}]{\abs{b_v - \frac{\lceil \theta kn \rceil}{kn}} \ge \frac{\delta}{2}},
  \end{align}
where the last inequality holds for $k \geq \frac{2}{\delta}$.

Furthermore, by \Cref{lemma:hypergometric-concentration}, there exists  $K_0=K_0(\delta,\epsilon,n)$ such that for $k \ge K_0$,
  \begin{align}\label{eq:bad-2}
  \forall v \in V,\quad  \Pr[k\*b \sim \HyperGeo]{\abs{b_v - \frac{\lceil \theta kn \rceil}{kn}} \ge \frac{\delta}{2}} \leq  2 \exp \tp{\frac{-\delta^2 k }{2} } \le \frac{\epsilon}{n},
  \end{align}
  where the last inequality holds because $k$ is sufficiently large.
  %Note that $\delta = \delta(\mu, \theta, \epsilon)$ and $n$ is determined by $\mu$.
  Combining ~\eqref{eq:bad-1} and ~\eqref{eq:bad-2} proves the claim.	
\end{proof}

\section{Implication to Modified Log-Sobolev Constant}
\label{section:compare-FD-GD}

In this section, we prove \Cref{lemma-imply-mls}.
Let $P = \PGD_\mu$ denote the Glauber dynamics on $\mu$.
We only need to prove the following lemma.
\begin{lemma} \label{lem:compare}
Let $\theta \in (0, 1)$ and $\mu$ be a distribution over $\{\0, \1\}^V$. %and $\theta \in (0, 1)$ be a real number.
  For any function $f: \{\0, \1\}^V \to \mathds{R}_{\geq 0}$,
  \begin{align} \label{eq:compare-target}
  \rho^{\mathrm{GD}}_{\min}(\pi) \cdot\frac{Z_\pi}{\theta^{n}} \E[R \sim \Bin(V,1-\theta)]{\pi_{R}\tp{\mathds{1}_{R}} \cdot \Ent[\pi^{\mathds{1}_{R}}]{f}} \leq \+E_{P}(f, \log f),
  \end{align} 
  where $n = |V|$, $\pi = \mu^{\tp{\*\theta}}$, $ \rho^{\mathrm{GD}}_{\min}(\pi)$ is defined in~\eqref{eq-def-MLS-min}, and $Z_\pi$ is defined in~\eqref{eq-redef-pi}.
\end{lemma}
\Cref{lemma-imply-mls} is a straightforward  consequence of \Cref{lem:compare}. The rest of this section is dedicated to the proof of \Cref{lem:compare}.

For convenience, we define the following notation.
%\begin{definition} \label{def:mEnt} % modified entropy
  Suppose $\mu$ is a distribution over $\{\0, \1\}^V$ and $f: \{\0, \1\}^V \to \mathds{R}_{\geq 0}$ a function, we define the \emph{covariance} between $f$ and $\log f$ with respect to $\mu$ as
  \begin{align*}
    \MEnt[\mu]{f} &\triangleq \frac{1}{2} \sum_{\sigma, \tau \in \Omega(\mu)} \mu(\sigma) \mu(\tau) \tp{f(\sigma) - f(\tau)} \tp{\log f(\sigma) - \log f(\tau)}.
  \end{align*}
%\end{definition}
Note that for any $f:\{\0,\1\}^V \to \mathds{R}_{\geq 0}$, it holds that $\MEnt[\mu]{f} \geq 0$.
%Note that in this definition, it always holds that $\tp{f(\sigma) - f(\tau)}\tp{\log f(\sigma) - \log f(\tau)} \geq 0$, so we have $\MEnt[\mu]{f} \geq 0$ for any $f: \{\0, \1\}^V \to \mathds{R}_{\geq 0}$.

Let $\mu$ be a distribution over $\{\0, \1\}^V$ and let $C > 0$ be a constant.
Let $P$ denote the Glauber dynamics for $\mu$.
For function $f: \{\0, \1\}^V \to \mathds{R}_{\geq 0}$, the Dirichlet form satisfies
\begin{align}
  \nonumber \+E_P(f,\log f) &= \frac{1}{2}\sum_{\sigma, \tau \in \Omega(\mu)} \mu(\sigma) P(\sigma, \tau) \tp{f(\sigma) - f(\tau)} \tp{\log f(\sigma) - \log f(\tau)} \\
  \nonumber &= \frac{1}{2}\sum_{\sigma \in \Omega(\mu)}\mu(\sigma) \frac{1}{n} \sum_{v \in V} \sum_{\tau \in \Omega(\mu^{\sigma_{V\setminus v}})}\mu^{\sigma_{V\setminus v}}(\tau) \tp{f(\sigma) - f(\tau)} \tp{\log f(\sigma) - \log f(\tau)} \\
  \nonumber &= \frac{1}{n} \sum_{v\in V} \sum_{\chi \in \Omega(\mu_{V\setminus \{v\}})} \mu_{V\setminus \{v\}}(\chi) \cdot  \frac{1}{2} \sum_{\sigma, \tau \in \Omega(\mu^\rho)} \mu^\chi(\sigma)\mu^\chi(\tau) \tp{f(\sigma) - f(\tau)} \tp{\log f(\sigma) - \log f(\tau)} \\
  \label{eq:mAT} &= \frac{1}{n} \sum_{v\in V} \sum_{\chi \in \Omega(\mu_{V\setminus \{v\}})} \mu_{V\setminus \{v\}}(\chi)  \cdot \MEnt[\mu^\chi]{f} = \frac{1}{n}\sum_{v\in V} \mu[\MEnt[v]{f}], % modified approximate tensorization
\end{align}
where we  use the following notation
\begin{align*}
  \mu[\MEnt[v]{f}] &\triangleq \sum_{\chi \in \Omega(\mu_{V\setminus \{v\}})} \mu_{V\setminus \{v\}}(\chi) \cdot \MEnt[\mu^\chi]{f}.
\end{align*}
Recall that we say $P$ satisfies the \emph{modified log-Sobolev inequality} with constant $C$ if for any $f: \{\0, \1\}^V \to \mathds{R}_{\geq 0}$, it holds that
\begin{align} \label{eq:mAT-short}
  C \cdot \Ent[\mu]{f} \leq  \+E_P(f,\log f) = \frac{1}{n}\sum_{v\in V} \mu[\MEnt[v]{f}].
\end{align}

Now, we are ready to prove \Cref{lem:compare}
%\subsection{Proof of \Cref{lem:compare}}
Consider the modified log-Sobolev inequality for the Glauber dynamics on  $\pi^{\mathds{1}_R}$. 
By \eqref{eq:mAT-short}, for all $R\subseteq V$ with $\pi_R(\mathds{1}_R)>0$,  the Glauber dynamics on $\pi^{\mathds{1}_R}$ satisfies the modified log-Sobolev inequality with constant $\rho^{\mathrm{GD}}(\pi^{\mathds{1}_R})$:
\[
\forall f\in\mathds{R}^{\Omega(\mu)},\quad
\rho^{\mathrm{GD}}_{\min}(\pi) \cdot \Ent[\pi^{\mathbf{1}_R}]{f}
\le
\rho^{\mathrm{GD}}(\pi^{\mathds{1}_R}) \cdot \Ent[\pi^{\mathbf{1}_R}]{f}
\le
\frac{1}{n}\sum_{v \in V} \pi^{\mathds{1}_R}[\MEnt[v]{f}],
\]
where $\rho^{\mathrm{GD}}_{\min}(\pi) \leq \rho^{\mathrm{GD}}(\pi^\sigma)$ for all feasible partial configuration $\sigma$, as defined in~\eqref{eq-def-MLS-min}.
Recall that $\Bin(V,1-\theta)$ denotes the distribution of subset $R \subseteq V$ that is randomly generated by including each $v \in V$ into $R$ independently with probability $1 - \theta$.
The LHS of \eqref{eq:compare-target} can be upper bounded as follows
\begin{align}
\label{eq:FD-Dirichlet-upper-bound}
  &\rho^{\mathrm{GD}}_{\min}(\pi) \cdot \frac{Z_\pi}{\theta^{\abs{V}}} \E[R \sim \Bin(V,1-\theta)]{\pi_R(\mathds{1}_R) \cdot \Ent[\pi^{\mathds{1}_R}]{f}}\notag\\ 
  \leq\,& \frac{Z_\pi}{\theta^{\abs{V}}} \cdot \frac{1}{n} \sum_{v \in V} \E[R \sim \Bin(V,1-\theta)]{\pi_R(\mathds{1}_R)\cdot\pi^{\mathds{1}_R}[\MEnt[v]{f}]}.
\end{align}

Let $I[\cdot]$ denote the indicator variable.
The following identity holds for all $f\in\mathds{R}^{\Omega(\mu)}$:
\begin{align}
\E[R \sim \Bin(V,1-\theta)]{\pi_R(\mathds{1}_R)\cdot\pi^{\mathds{1}_R}[\MEnt[v]{f}]}
&\overset{(\star)}{=}
\E[R \sim \Bin(V,1-\theta)]{\E[\sigma\sim\pi_{V\setminus\{v\}}]{I\left[R\subseteq\sigma^{-1}(\1)\right]\cdot\MEnt[\pi^{\sigma}]{f}}}\notag\\
&= \E[\sigma\sim\pi_{V\setminus\{v\}}]{\Pr[R \sim \Bin(V,1-\theta)]{R\subseteq\sigma^{-1}(\1)}\cdot\MEnt[\pi^{\sigma}]{f}}\notag\\
&= \E[\sigma\sim\pi_{V\setminus\{v\}}]{\theta^{|V|-\|\sigma\|_{+}}\cdot\MEnt[\pi^{\sigma}]{f}},\label{eq:entropy-expectation-R-identity}
\end{align}
where the nontrivial equation $(\star)$ holds by verifying for every choice of $v\in V$ and $R\subseteq V$ as follows: 
\begin{itemize}
\item For the case that $\pi_R(\mathds{1}_R)>0$ and $v\not\in R$, it holds that
\[
\pi_R(\mathds{1}_R)=\Pr[\sigma\sim\pi]{R\subseteq\sigma^{-1}(\1)}=\Pr[\sigma\sim\pi_{V\setminus\{v\}}]{R\subseteq\sigma^{-1}(\1)},
\]
and $\pi^{\mathds{1}_R}[\MEnt[v]{f}]$ is well-defined, such that
\[
\pi^{\mathds{1}_R}[\MEnt[v]{f}]=\E[\sigma\sim\pi_{V\setminus\{v\}}^{\mathds{1}_R}]{\MEnt[\pi^{\sigma}]{f}}=\E[\sigma\sim\pi_{V\setminus\{v\}}]{\MEnt[\pi^{\sigma}]{f}\mid R\subseteq\sigma^{-1}(\1)}.
\]
Therefore, 
\begin{align*}
\pi_R(\mathds{1}_R)\cdot\pi^{\mathds{1}_R}[\MEnt[v]{f}]
&=
\Pr[\sigma\sim\pi_{V\setminus\{v\}}]{R\subseteq\sigma^{-1}(\1)}
\cdot
\E[\sigma\sim\pi_{V\setminus\{v\}}]{\MEnt[\pi^{\sigma}]{f}\mid R\subseteq\sigma^{-1}(\1)}\\
&=
\E[\sigma\sim\pi_{V\setminus\{v\}}]{I\left[R\subseteq\sigma^{-1}(\1)\right]\cdot \MEnt[\pi^{\sigma}]{f}}.
\end{align*}
\item For the case that $\pi_R(\mathds{1}_R)=0$ or $v\in R$, both sides are 0. 
On the left-hand-side, if $\pi_R(\mathds{1}_R)=0$, then by convention 
\[
\pi_R(\mathds{1}_R)\cdot\MEnt[\pi^{\mathds{1}_R}]{f}=0;
\] 
or else, if $\pi_R(\mathds{1}_R)>0$ but $v\in R$, then  $\pi^{\mathds{1}_R}[\MEnt[v]{f}]$ is well-defined, but for $\sigma\sim\pi_{V\setminus\{v\}}^{\mathds{1}_R}$, 
the $\mathds{1}_R\uplus\sigma$ gives a configuration fully specified on $V$ and hence $|\Omega(\pi^{\mathds{1}_R\uplus\sigma})| = 1$. In this case, by definition,  the covariance must be 0, i.e. 
\[
\pi^{\mathds{1}_R}[\MEnt[v]{f}]=\E[\sigma\sim\pi_{V\setminus\{v\}}^{\mathds{1}_R}]{\MEnt[\pi^{\mathds{1}_R\uplus\sigma}]{f}}=0.
\]
On the right-hand-side, if $\pi_R(\mathds{1}_R)=0$ or $v\in R$, then for $\sigma\sim\pi_{V\setminus\{v\}}$, the event $R\subseteq\sigma^{-1}(\1)$ can never occur, and hence
\[
\E[\sigma\sim\pi_{V\setminus\{v\}}]{I\left[R\subseteq\sigma^{-1}(\1)\right]\cdot \MEnt[\pi^{\sigma}]{f}}=0.
\]
\end{itemize}
This gives the equation $(\star)$ in~\eqref{eq:entropy-expectation-R-identity}.
Meanwhile, the other two equations in~\eqref{eq:entropy-expectation-R-identity} follows respectively from  linearity of expectation and the fact that $\Pr[R \subseteq V]{R\subseteq \Lambda}=\theta^{|V|-|\Lambda|}$ for all $\Lambda\subseteq V$.

Furthermore, it can be verified that
\begin{align}
  &\E[\sigma \sim \pi_{V \setminus\{v\}}]{\frac{1}{\theta^{\norm{\sigma}_+}}\MEnt[\pi^\sigma]{f}}\notag\\
  =\,& \sum_{\sigma \in \Omega(\pi_{V\setminus \{v\}})} \frac{1}{\theta^{\norm{\sigma}_+}}\pi_{V\setminus\{v\}}(\sigma)  \pi^\sigma_v(-1)  \pi^\sigma_v(+1)  (f(\sigma_+^v) - f(\sigma_{-}^v)) \tp{\log f(\sigma_+^v) - \log f(\sigma_-^v)} \notag\\
  =\,& \frac{1}{Z_\pi} \sum_{\sigma \in \Omega(\pi_{V\setminus \{v\}})} \mu_{V\setminus\{v\}}(\sigma) \mu^\sigma_v(-1)  \pi^\sigma_v(+1)   (f(\sigma_+^v) - f(\sigma_{-}^v)) \tp{\log f(\sigma_+^v) - \log f(\sigma_-^v)} \notag\\
  \leq\,& \frac{1}{Z_\pi} \sum_{\sigma \in \Omega(\pi_{V\setminus \{v\}})} \mu_{V\setminus\{v\}}(\sigma) \mu^\sigma_v(-1)  \mu^\sigma_v(+1)  (f(\sigma_+^v) - f(\sigma_{-}^v)) \tp{\log f(\sigma_+^v) - \log f(\sigma_-^v)}\notag\\
  =\,& \frac{1}{Z_\pi} \mu[\MEnt[v]{f}], \label{eq:tensorization-change-base}
\end{align}
where  $\sigma_{\pm}^v\in\{\0,\1\}^V$ denote the configurations on $V$ where $\sigma_{\pm}^v(V\setminus\{v\})=\sigma(V \setminus \{v\} )$ and $\sigma_{\pm}^v(v)=\pm1$,
the second equation is due to the chain rule and the relation between $\pi$ and $\mu$ in~\eqref{eq-redef-pi}, and the last inequality is due to $\Omega(\pi)=\Omega(\mu)$, $(f(\sigma_+^v) - f(\sigma_{-}^v)) \tp{\log f(\sigma_+^v) - \log f(\sigma_-^v)}\geq 0$ and the relaxation 
\begin{align}
\pi_{v}^{\sigma}(\1)\le\mu_v^{\sigma}(\1),\label{eq:margin-monotone-pi-mu}
\end{align}
which holds because $\pi=\mu^{\tp{\*\theta}}$ is obtained by biasing every variable with a local field $\theta\in(0,1)$.
Indeed, 
\begin{align*}
\pi^\sigma_v(\1) = \frac{\pi(\sigma_{+})}{\pi(\sigma_{-}) + \pi(\sigma_{+}) }	= \frac{\theta \mu(\sigma_{+})}{\mu(\sigma_{-}) + \theta\mu(\sigma_{+}) } \leq \frac{\mu(\sigma_{+})}{\mu(\sigma_{-}) + \mu(\sigma_{+}) } = \mu^\sigma_v(\1),
\end{align*}
where the inequality holds for $ \theta \in(0, 1)$.

Combining~\eqref{eq:FD-Dirichlet-upper-bound}, \eqref{eq:entropy-expectation-R-identity}, and~\eqref{eq:tensorization-change-base}, we have
\begin{align*}
\rho^{\mathrm{GD}}_{\min}(\pi) \cdot \frac{Z_\pi}{\theta^{\abs{V}}} \E[R \sim \Bin(V,1-\theta)]{\pi_R(\mathds{1}_R) \cdot \Ent[\pi^{\mathds{1}_R}]{f}}	 \leq 
\frac{1}{n}\sum_{v \in V} \mu[\MEnt[v]{f}]=\+E_{P}(f,\log f),
\end{align*}
where $P$ denotes the Glauber dynamics for $\mu$. This proves \Cref{lem:compare}.

\section{Application in Ising Model}
\label{section-Ising}
In this section, we prove \Cref{thm:main}.
%We introduce the general Ising model with local fields. 
%%
%Let $G=(V,E)$ be a graph, $\beta \in \mathbb{R}_{>0}$ be the edge activity and $\*\lambda_v = (\lambda_v)_{v \in V} \in \mathbb{R}_{>0}^V$ be the local fields.
%%
%The Gibbs distribution $\mu$ is defined by
%\begin{align*}
%\forall \sigma \in \{\0,\1\}^V,\quad \mu(\sigma) \propto \beta^{m(\sigma)} \prod_{v\in V: \sigma_v = \1}\lambda_v,
%\end{align*}
%where $m(\sigma) =\abs{{\{u,v\}\in E \mid \sigma_u  = \sigma_v}}$ denotes the number of monochromatic edges.
Let $G=(V,E)$ be a graph, $\beta \in \mathds{R}_{>0}$ and $\*\lambda = (\lambda_v)_{v \in V} \in \mathds{R}_{>0}^V$.
We use the notation $\+I = (G,\beta,\*\lambda)$ to denote an Ising model and $\mu = \mu_{\+I}$ to denote the Gibbs distribution defined by $\+I$.

%\todo{We can now only handle the case where external fields are unified.}
%\todo{the following parts should be moved to section 2}
In order to apply \Cref{theorem-general}, we introduce the flipping operation in \cite{chen2021rapid}.
\begin{definition}
Let $\mu$ be a distribution over $\{\0,\1\}^V$, $\*\chi \in \{\0,\1\}^V$ be a direction vector.
The flipped distribution $\nu = \mathrm{flip}(\mu,\*\chi)$ over $\{\0,\1\}^V$ is defined as
\begin{align*}
	\forall \sigma \in \{\0,\1\}^V, \quad \nu(\sigma) = \mu(\sigma \odot \*\chi),
\end{align*} 
where $(\sigma \odot \*\chi)_v = \sigma_v \chi_v$ for all $v \in V$.
\end{definition}

%The flipped distribution $\nu$ is closely related to the original distribution $\mu$:
The following fact is straightforward to verify.
\begin{fact}\label{fact:flip}
  For any distribution $\mu$ over $\{\0,\1\}^V$, $\*\chi \in \{\0,\1\}^V$ and $\nu = \mathrm{flip}(\mu,\*\chi)$, it holds that
  \begin{enumerate}
    \item $\rho^{\mathrm{GD}}(\nu) = \rho^{\mathrm{GD}}(\mu)$, $\rho^{\mathrm{GD}}_{\min}(\nu) = \rho^{\mathrm{GD}}_{\min}(\mu)$, where $\rho^{\mathrm{GD}}(\cdot)$ is the modified log-Sobolev constant of Glauber dynamics, and $\rho^{\mathrm{GD}}_{\min}(\cdot)$ is defined in~\eqref{eq-def-MLS-min};
	 \item for any $\*\phi \in \mathbb{R}_{>0}^V$, $\nu^{(\*\phi)} = \mathrm{flip}(\mu^{\tp{\*\phi^{\*\chi}}},\*\chi)$, where $\*\phi^{\*\chi} \in \mathbb{R}_{>0}^V$ satisfying $(\*\phi^{\*\chi})_v = \phi_v^{\chi_v}$ for all $v \in V$;
    \item $\sup_{\*\phi \in \mathbb{R}_{>0}^V} \norm{\Psi^{\mathrm{Inf}}_{\nu^{(\*\phi)}}}_{\infty} = \sup_{\*\phi \in \mathbb{R}_{>0}^V} \norm{\Psi^{\mathrm{Inf}}_{\mu^{(\*\phi)}}}_{\infty} $.
  \end{enumerate}
\end{fact}

Given an Ising model $\+I=(G,\beta,\*\lambda)$, let $\*\chi_{\+I} \in \{\0,\1\}^V$ be the direction vector defined as follows:
\begin{align}
\label{eq-def-chi}
\forall v \in V, \quad \chi_v &= \begin{cases}
\1,& \lambda_v \geq 1\\
\0,& \lambda_v < 1
\end{cases}
\end{align}
We prove that the flipped distribution $\nu=\mathrm{flip}(\mu,\*\chi)$ is spectrally independent with all fields, and with proper local fields, the MLS constant is easy-to-analyze. 
\begin{lemma}\label{thm:verification}
Let  $\delta, C \in (0,1)$.
For any graph $G=(V,E)$, any $\beta \in \mathbb{R}_{>0}$ and any $\*\lambda=(\lambda_v)_{v \in V} \in \mathbb{R}^V_{>0}$, if $\beta \in \left[\frac{\Delta-2+\delta}{\Delta - \delta},\frac{\Delta-\delta}{\Delta-2+\delta}\right]$, where $\Delta \geq 3$ is the maximum degree of $G$,  let $\mu = \mu_{\+I}$ denote the Gibbs distribution defined by the Ising model $\+I=(G,\beta,\*\lambda)$, %the following results hold.
the flipped distribution $\nu = \mathrm{flip}(\mu,\*\chi_{\+I})$ satisfies
  \begin{enumerate}
    \item 
 	$\sup_{\*\phi \in \mathbb{R}_{>0}^V} \norm{\Psi^{\mathrm{Inf}}_{\nu^{(\*\phi)}}}_{\infty} = \sup_{\*\phi \in \mathbb{R}_{>0}^V} \norm{\Psi^{\mathrm{Inf}}_{\mu^{(\*\phi)}}}_{\infty} \leq \frac{2}{\delta}$;
%    \item If $\lambda_v = \lambda \in \tp{\frac{1}{500},1}$ for all $v \in V$, then $\rho^{\mathrm{GD}}_{\min}(\pi) \ge \frac{10^{-8}}{n}$ where $\pi = \mu^{\tp{\frac{1}{500}}}$.
    \item 
    %Let $\*\chi_{\+I}\in \{\0,\1\}^V$ denote the direction vector in~\eqref{eq-def-chi} and $\nu = \mathrm{flip}(\mu,\*\chi_{\+I})$.
    If $\min\tp{\lambda_v,\frac{1}{\lambda_v}} \ge C$ for all $v \in V$, then $\rho^{\mathrm{GD}}_{\min}( \nu^{\tp{\frac{1}{500}}}) \ge \frac{10^{-6}C}{n}$.% where $\pi =$.
  \end{enumerate}
\end{lemma}

%Together with those facts, it is straightforward to prove \Cref{thm:main}.
We are ready to prove \Cref{thm:main}.
\begin{proof}[Proof of \Cref{thm:main}]
%Recall that the input Ising model is defined on graph $G$ with edge activity $\beta$ and uniform external filed $\lambda$.
%  Without loss of generality, we may assume that $\lambda \le 1$, otherwise we can flip the role of $\0$ and $\1$.
%  If $\lambda < \frac{1}{500}$, classical coupling analysis yields an $O(n \log n)$ mixing time. 
%  Therefore, we may further assume $\lambda \in \left[\frac{1}{500},1\right]$. 
%  By the first part of \Cref{thm:verification}, we know that such Ising model has bounded total influence under arbitrary local fields.
%  We use \Cref{theorem-general} with $\theta=\frac{1}{500}$. By the second part of \Cref{thm:verification}, the MLS constant for the Glauber dynamics on $\mu$ is lower bounded by
%  \begin{align*}
%    \rho^{\mathrm{GD}}(\mu) \ge \tp{\frac{1}{1000}}^{4/\delta + 7} \cdot \frac{10^{-8}}{n} \geq \frac{10^{-\frac{12}{\delta}-30}}{n}. 
%  \end{align*}
Fix an Ising model $\+I = (G,\beta, \*\lambda)$ with Gibbs distribution $\mu$. 
Consider the Glauber dynamics on $\mu$.
Note that if $\min\tp{\lambda_v,\frac{1}{\lambda_v}} \leq \frac{1}{500}$ for all $v \in V$, then the classic path coupling~\cite{bubley1997path} yields an $O(n \log n)$ mixing time.
Therefore, we may assume that there exists a vertex $v^\star \in V$ satisfying $\min \tp{\lambda_{v^\star},\frac{1}{\lambda_{v^\star}}} \ge \frac{1}{500}$, which implies
$\min \tp{\lambda_v,\frac{1}{\lambda_v}} \ge \frac{\lambda_{\min}}{500 \lambda_{\max}}$ for all $v \in V$ and $\lambda_{\min} \leq 500$, $\lambda_{\max} \geq \frac{1}{500}$.

Let $\nu = \mathrm{flip}(\mu,\*\chi_{\+I})$, where $\*\chi_{\+I}$ is defined in~\eqref{eq-def-chi}.
The first part of \Cref{fact:flip} shows that $\nu$ is $\frac{2}{\delta}$-spectrally independent with all fields. 
The second part shows that 
$\rho^{\mathrm{GD}}_{\min}( \nu^{\tp{\frac{1}{500}}}) \ge \frac{10^{-6}C}{n}$ where $C = \frac{\lambda_{\min}}{500 \lambda_{\max}}$.
Using \Cref{theorem-general} with $\theta = \frac{1}{500}$, the Glauber dynamics on $\nu$ has  MLS constant
 \begin{align*}
  \rho^{\mathrm{GD}}(\nu) \ge \tp{\frac{10^{-\frac{12}{\delta}-30} \lambda_{\min}}{\lambda_{\max}}} \frac{1}{n}.
\end{align*}

Furthermore, since $\lambda_{\min} \leq 500$ and $\lambda_{\max} \geq \frac{1}{500}$, we have
 \begin{align*}
  \nu_{\min} = \mu_{\min} \geq \tp{\min_{v \in V} \min_{0 \le s \le \Delta} \min\tp{\frac{\lambda_v \beta^{s}}{\lambda_v \beta^{s} + \beta^{\deg(v)-s}},\frac{\beta^{\deg(v)-s}}{\lambda_v \beta^s + \beta^{\deg(v)-s}}}}^n \ge \tp{\frac{\lambda_{\min}}{14000 \lambda_{\max}}}^n.
 \end{align*}
%Fix a local fields vector $\*\phi \in \mathds{R}_{>0}^V$.
%
%By \Cref{fact:flip}, it holds that $\nu^{(\*\phi)} = \mathrm{flip}(\mu^{\tp{\*\phi^{\*\chi_{\+I}}}},\*\chi_{\+I})$. 
%
%Consider an other Ising model $\+J = (G,\beta, \*\lambda' )$, where $\lambda'_v = (\*\lambda \odot \*\phi^{\*\chi_{\+I}})_v = \lambda_v \phi_v^{\chi_{\+I}(v)}$ for all $v \in V$.
%
%Let $\pi$ denote the Gibbs distribution defined by $\+J$. 
%
%It holds that $\nu^{(\*\phi)} = \mathrm{flip}(\pi,\*\chi_{\+I})$.
%
%Using the first part of \Cref{thm:verification} with Gibbs distribution $\pi$ and direction vector  $\*chi_{\+I}$, we know that $\nu^{(\*\phi)}$ has $\frac{2}{\delta}$-bounded total influence. 

 By \Cref{fact:flip}, $\rho^{\mathrm{GD}}(\nu) = \rho^{\mathrm{GD}}(\mu)$.
 Note that the transition matrix of Glauber dynamics has non-negative eigenvalues~\cite{DGU14,alev2020improved}.
By~\eqref{eq:MLS-mixing}, the mixing time of the Glauber dynamics can be bounded by
 \begin{align*}
T_{\mathrm{mix}}(\epsilon) &\leq  \frac{1}{\rho^{\mathrm{GD}}(\mu)} \tp{\log \log \frac{1}{\mu_{\mathrm{min}}} + \log \frac{1}{2 \epsilon^2}}	\leq \frac{\lambda_{\max}}{\lambda_{\min}} 10^{30 + \frac{12}{\delta}} n \tp{\log n + \log \log \frac{14000\lambda_{\max}}{\lambda_{\min}} + \log \frac{1}{2\epsilon^2}}\\
&=\exp\tp{O(1/\delta)} \frac{\lambda_{\max}}{\lambda_{\min}} n \tp{\log \frac{n}{\epsilon} + \log \frac{2\lambda_{\max}}{\lambda_{\min}}}.
 \end{align*}
This proves the theorem.
\end{proof}

%\todo{end of content}

%We list several cutting-edge results on both spectral independence and MLS constant for Ising model.
We need following two lemmas to prove \Cref{thm:verification}.
\begin{lemma}[\text{\cite{chen2020optimal}}]\label{thm:CLV-spectral}
  Let $\mu$ be the Gibbs distribution of Ising model on graph $G=(V,E)$ with maximum degree $\Delta \ge 3$, edge activity $\beta$ and local fields $\*\lambda = (\lambda_v)_{v \in V}$. If there exists $\alpha \in (0,1)$ such that
  \begin{align}\label{eq:one-potential}
    h(y) \triangleq \abs{\frac{(1-\beta^2) e^y}{(\beta e^y+1)(\beta+e^y)}} \le \frac{1-\alpha}{\Delta-1},  
  \end{align}
  then it holds that  $\Vert \Psi^{\mathrm{inf}}_\mu \Vert_\infty \leq \frac{2}{\alpha}$.
  %$\mu$ is $\frac{2}{\alpha}$-spectrally independent.
\end{lemma}

\begin{lemma}[\text{\cite{Katalin19,HA20}}]\label{thm:MLS-biased}
  Let $\mu$ be a distribution with support $\{\0,\1\}^V$ and $P^{\mathrm{GD}}_{\mu}$ be the Glauber dynamics on $\mu$. 
  Let $\alpha = \min_{v \in V} \min_{\sigma \in \{\0,\1\}^V} \mu_v^{\sigma_{V \setminus \{v\}}}(\sigma_v)$ denote the marginal lower bound. 
  Furthermore, denote the Dobrushin's influence matrix of $\mu$ by $A=(A_{u,v})_{u,v\in V}$, which is defined as
\begin{align*}
A_{u,v}= \begin{cases}
\max_{(\sigma,\tau) \in S(u,v)}\DTV{\mu^\sigma_v}{\mu^\tau_v} &\text{if } u \neq v\\
0 &\text{if } u = v,
\end{cases}
\end{align*}
where $S(u,v)$ contains all pairs of configurations $\sigma,\tau \in \{\0,\1\}^{V \setminus \{v\} }$ that disagree only at vertex $u$.
%  \begin{align*}
%    A_{u,v}=
%    \begin{cases}
%      \max_{\substack{\sigma,\tau \in \{\0,\1\}^V \\ \sigma_{V \setminus \{u\} } = \tau_{V \setminus \{u\} } }} \abs{\mu^{\sigma_{V \setminus \{v\}}}_v(\sigma_v) - \mu^{\tau_{V \setminus \{v\} }}(\tau)} & ,u \neq v \in V,\\
%      0 & ,\text{otherwise.}
%    \end{cases}
%  \end{align*}
Let $|V| = n$.
  If $\norm{A}_2 <1$, then the Glauber dynamics on $\mu$ has modified log-Sobolev constant
  \begin{align*}
  \rho^{\mathrm{GD}}(\mu) \geq \frac{\alpha(1-\norm{A}_2)^2}{2n}.	
  \end{align*}
   %for all distribution $\nu$ over $\{\0,\1\}^V$,
%  \begin{align*}
%    D(\nu P^{\mathrm{GD}}_{\mu} \| \mu ) \le \tp{1-\frac{\alpha(1-\norm{A}_2)^2}{2n}} D(\nu \| \mu).
%  \end{align*}
\end{lemma}

We are now ready to present the proof of \Cref{thm:verification}
\begin{proof}[Proof of \Cref{thm:verification}]
  For any $\*\lambda \in \mathbb{R}_{>0}^V$, note that when $\beta \in \left[\frac{\Delta-2+\delta}{\Delta-\delta},\frac{\Delta-\delta}{\Delta-2+\delta}\right]$,
  \begin{align*}
    \abs{\frac{(1-\beta^2)e^y}{(\beta e^y+1)(\beta +e^y)}} \le \frac{\abs{1-\beta}}{1+\beta} \le \frac{1-\delta}{\Delta-1}. 
  \end{align*}
  Together with \Cref{thm:CLV-spectral} and \Cref{fact:flip}, we have 
$
\sup_{\*\phi \in \mathbb{R}_{>0}^V}\Vert \Psi^{\mathrm{inf}}_{\nu^{(\*\phi)}} \Vert_\infty = \sup_{\*\phi \in \mathbb{R}_{>0}^V} \Vert \Psi^{\mathrm{inf}}_{\mu^{(\*\phi)}} \Vert_\infty \leq \frac{2}{\delta}.
$

Let $\*\chi = \*\chi_{\+I}$. Recall that $\nu^{\tp{\frac{1}{500}}} = \mathrm{flip}\tp{\mu^{\tp{\frac{1}{500}}^{\*\chi}},\*\chi}$ and $\rho^{\mathrm{GD}}_{\min}\tp{\nu^{\tp{\frac{1}{500}}}} = \rho^{\mathrm{GD}}_{\min} \tp{\mu^{\tp{\frac{1}{500}}^{\*\chi}}}$.
Let $\pi \triangleq \mu^{\tp{\frac{1}{500}}^{\*\chi}}$.
Note that $\pi$ is the Gibbs distribution of the Ising model with external field $\*\lambda^\star \in \mathbb{R}_{>0}^V$, where 
\begin{align*}
\lambda^\star_v =
\begin{cases}
500\lambda_v & \lambda_v \geq 1;\\
\frac{\lambda_v}{500}, & \lambda_v < 1.
\end{cases}
\end{align*}
It suffices to prove that 
\begin{align*}
	\rho^{\mathrm{GD}}_{\min}\tp{\pi} \ge \frac{10^{-6} C}{n}.
\end{align*}

Fix a subset $\Lambda \subseteq V$ of size $m = |V|$.
Fix a partial configuration $\tau \in \{\0,\1\}^{V \setminus \Lambda}$.
We analyze the MLS constant for the Glauber dynamics $P^{\mathrm{GD}}_{\pi^\tau}$ on $\pi^\tau$.
Note that $\pi^\tau$ is a distribution over $\{\0,\1\}^V$, where the configuration on $V \setminus \Lambda$ is fixed to $\tau$.
We consider the distribution $\pi^\tau_{\Lambda}$ and the Glauber dynamics $P^{\mathrm{GD}}_{\pi^\tau_\Lambda}$.
The following relation of MLS constants is straightforward to verify
\begin{align}
\label{eq-relation-condition-original}
\rho^{\mathrm{GD}}(\pi^\tau) = \frac{m}{n}\rho^{\mathrm{GD}}(\pi^\tau_\Lambda).	
\end{align}

We use \Cref{thm:MLS-biased} to analyze the MLS constant for Glauber dynamics on $\pi^\tau_\Lambda$.
Note that $\frac{C}{500} \le \min\tp{\lambda^\star_v,\frac{1}{\lambda^\star_v}} \le \frac{1}{500}$ holds for all $v \in V$. Let $\alpha$ be the marginal lower bound for $\pi^\tau_\Lambda$. It holds that 
  \begin{align*}
    \alpha &\geq \min_{v \in V} \min_{\sigma \in \{\0,\1\}^V} \pi_v^{\sigma_{V \setminus \{v\}}}(\sigma_v)\\
    &\ge \min_{v \in V} \min_{0 \le s \le d\le \Delta} \min \tp{\frac{\lambda^\star_v \beta^s}{\lambda^\star_v \beta^s + \beta^{d-s}}, \frac{\beta^{d-s}}{\lambda^\star_v \beta^s + \beta^{d-s}}}\\
    &\overset{(*)}{\ge} \min \tp{ \frac{\lambda^\star_v}{\lambda^\star_v+27}, \frac{1}{27\lambda^\star_v+1}}\\
    & \ge \frac{C}{2 \times 10^{4}}, 
  \end{align*}
  where $(*)$ follows from $\max_{-\Delta \le s \le \Delta} \beta^{s} \le \tp{1+\frac{2}{\Delta-2}}^\Delta \le 27$ when $\Delta \ge 3$. 
  Moreover, let $A$ denote the  Dobrushin's influence matrix for $\pi^\tau_\Lambda$. It holds that
  \begin{align*}
    \max \tp{\norm{A}_1,\norm{A}_{\infty}} &\le \Delta \sup_{\substack{\lambda \in \mathbb{R}_{>0}\\ \min(\lambda, \frac{1}{\lambda}) \le \frac{1}{500} }}\max_{0 \le s < d \le \Delta} \abs{\frac{\lambda \beta^s}{\lambda \beta^s + \beta^{d-s}} - \frac{\lambda \beta^{s+1}}{\lambda \beta^{s+1}+\beta^{d-s-1}}}\\
    &= \Delta \sup_{\substack{\lambda \in \mathbb{R}_{>0}\\ \min(\lambda, \frac{1}{\lambda}) \le \frac{1}{500} }} \max_{0 \le s < d \le \Delta} \abs{\frac{\lambda \beta^{d-1}(1-\beta^2)}{\lambda \beta^{d-1}(\beta^2+1)+\lambda^2 \beta^{2s+1} + \beta^{2d-2s-1}}}\\
    &\le \sup_{\substack{\lambda \in \mathbb{R}_{>0}\\ \min(\lambda, \frac{1}{\lambda}) \le \frac{1}{500} }} \max_{0 \le s < d \le \Delta} \min\tp{\frac{\lambda \Delta \abs{1-\beta^2}}{\beta^{d-2s}}, \frac{\Delta \abs{1-\beta^2}}{\lambda \beta^{2s-d+2}}}\\
    &\overset{(*)}{\le} \frac{3}{5},
  \end{align*}
  where $(*)$ follows from the fact that, for any $\lambda \in \mathbb{R}_{>0}$ satisfying $\min \tp{\lambda,\frac{1}{\lambda}} \le \frac{1}{500}$,
  \begin{align*}
    \max_{0 \le s < d \le \Delta}  \min\tp{\frac{\lambda \Delta \abs{1-\beta^2}}{\beta^{d-2s}}, \frac{\Delta \abs{1-\beta^2}}{\lambda \beta^{2s-d+2}}}
    \le 
    \begin{cases} 
      \min\tp{\lambda,\frac{1}{\lambda}} \Delta \tp{1-\beta^2} \beta^{-\Delta} \le 300  \min\tp{\lambda,\frac{1}{\lambda}} < \frac{3}{5}&, \beta \in \left(\frac{\Delta-2}{\Delta},1\right]\\
      \min\tp{\lambda,\frac{1}{\lambda}} \Delta \tp{\beta^2-1} \beta^{\Delta-2} \le 300  \min\tp{\lambda,\frac{1}{\lambda}} < \frac{3}{5} &, \beta \in \left[1,\frac{\Delta}{\Delta-2}\right).
    \end{cases}
  \end{align*}
%    \begin{align*}
%    \max_{0 \le s < d \le \Delta}  \min\tp{\frac{\lambda \Delta \abs{1-\beta^2}}{\beta^{d-2s}}, \frac{\Delta \abs{1-\beta^2}}{\lambda \beta^{2s-d+2}}}
%    \le 
%    \begin{cases} 
%      \min\tp{\lambda,\frac{1}{\lambda}} \Delta \tp{1-\tp{1-\frac{2}{\Delta}}^2} \beta^{-\Delta} \le 108 \lambda < \frac{1}{4}&, \beta \in \left(\frac{\Delta-2}{\Delta},1\right]\\
%      \min\tp{\lambda,\frac{1}{\lambda}} \Delta \tp{1-\beta^{-2}} \beta^{\Delta} \le 108 \lambda < \frac{1}{4} &, \beta \in \left[1,\frac{\Delta}{\Delta-2}\right).
%    \end{cases}
%  \end{align*}
  Hence, $\norm{A}_2 \le \sqrt{\norm{A}_1 \norm{A}_{\infty}} \le \frac{3}{5}$.
  Combining~\Cref{thm:MLS-biased} and~\eqref{eq-relation-condition-original}, we have
  \begin{align*}
  \rho^{\mathrm{GD}}(\pi^\tau) = \frac{m}{n}\rho^{\mathrm{GD}}(\pi^\tau_\Lambda) \geq 	\frac{m}{n} \cdot \frac{C\times (2/5)^2}{2\times 10^4 \times 2 m} \geq \frac{C}{10^6n}.
  \end{align*}
The above inequality holds for any $\tau$, which implies $\rho^{\mathrm{GD}}_{\min}\tp{\nu^{\tp{\frac{1}{500}}}} = \rho^{\mathrm{GD}}_{\min}\tp{\pi} \ge \frac{10^{-6} C}{n}$.
\end{proof}

\section{Acknowledgements}
Weiming Feng is supported by funding from the European Research Council (ERC) under the European Union's Horizon 2020 research and innovation programme (grant agreement No.~947778).

We would like to thank Heng Guo for helpful discussions.

\bibliographystyle{alpha}
\bibliography{refs.bib}

\end{document}